\documentclass[11pt,a4paper]{amsart}
\usepackage[active]{srcltx}
\usepackage{amsmath,amsfonts,amssymb,epsf,xypic,amsthm,mathrsfs,lmodern}
\usepackage[T1]{fontenc}
\usepackage[utf8]{inputenc}  

 \usepackage {hyperref}

\usepackage[lite,alphabetic]{amsrefs}

\usepackage[usenames]{xcolor} 

\theoremstyle{plain}
\newtheorem{theo}{Theorem}[section]
\newtheorem{lemma}[theo]{Lemma}
\newtheorem{prop}[theo]{Proposition}
\newtheorem{coro}[theo]{Corollary}

\theoremstyle{definition}

\theoremstyle{remark}
\newtheorem{rema}[theo]{Remark}

\def\cal{\mathcal}

\def\a{\alpha}
\def\b{\beta}

\def\g{\gamma}
\def\G{\Gamma}

\def\ad{{\rm ad}}
\def\id{{\rm id}}
\def\deg{{\rm deg}\,}
\def\PU{{\rm PU}}
\def\SU{{\rm SU}}
\def\GL{{\rm GL}}
\def\End{{\rm End}}
\def\Hom{{\rm Hom}}
\def\U{{\rm U}}
\def\SL{{\rm SL}}
\def\SO{{\rm SO}}
\def\rk{{\rm rk}}
\def\Ker{{\rm Ker\,}}
\def\Im{{\rm Im\,}}
\def\vol{{\rm vol}}
\def\Sp{{\rm Sp}}

\def\C{{\mathbb C}}

\def\R{{\mathbb R}}
\def\P{{\mathbb P}}

\def\fd{\longrightarrow}
\def\pfd{\rightarrow}

\def\la{\langle}
\def\ra{\rangle}
\def\ov{\overline}

\def\t{\theta}
\def\ds{\displaystyle}
\def\e{\varepsilon}

\def\o{\omega}
\def\mf{Y}

\def\Y{{\mathcal Y}}

\def\L{{\mathcal L}}
\def\F{{\mathcal F}}
\def\T{{\mathcal T}}
\def\S{{\mathcal S}}

\def\mg{{\mathfrak m}}

\def\hg{{\mathfrak h}}

\def\sg{{\mathfrak s}}

\def\r{\rho}

\def\tr{{\rm tr}\,}

\def\V{{\mathbb V}}
\def\W{{\mathbb W}}
\def\E{{\mathbb E}}

\def\s{\sigma}
\def\om{\omega}

\def\dego{{\rm deg}_{\mathcal T}}
\def\degog{{\rm deg}_{{\mathcal T}\hspace{-1pt},\,\mu_{\mathcal G}}\,}
\def\degO{{\rm deg}_{{\mathcal T}\hspace{-1pt},\,\mu}\,}
\def\into{\int_{{\mathcal T}\hspace{-1pt},\,\mu_{\mathcal G}}}
\def\intO{\int_{{\mathcal T}\hspace{-1pt},\,\mu}}

\def\stab{M}

\def\H{{{\mathbb H}^n_{\mathbb C}}}

\def\TH{T_{\H}}
\def\PTX{{{\mathbb P}T_X}}
\def\PTH{{\mathbb P}T_{{\mathbb H}^n_{\mathbb C}}}

\def\OX{{\mathcal O}_{\PTX}(-1)}

\def\F{{\mathcal F}}
\def\CP{{\mathbb{CP}}}
\def\s{{\mathfrak s}}
\def\u{{\mathfrak u}}
\def\su{{\mathfrak{su}}}

\def\Y{{\mathcal Y}}
\def\rmo{\sqrt{-1}}

\def\sing{{\mathcal S}(\mathcal F)}
\def\singsat{{S}}
\def\pig{\pi_{\cal G}}
\def\Omegag{\Omega_{\mathcal G}}

\begin{document}

\title[Maximal representations of uniform complex
  hyperbolic lattices]{Maximal representations of uniform complex
  hyperbolic lattices}

\author{Vincent Koziarz} 
\author{Julien Maubon}
\address[Vincent Koziarz]{Univ. Bordeaux, IMB, UMR 5251, F-33400 Talence, France}
\email{vkoziarz@math.u-bordeaux1.fr}
\address[Julien Maubon]{IECL, UMR 7502, Universit\'e
   de Lorraine, B. P. 70239, F-54506 Vand\oe uvre-l\`es-Nancy Cedex,
    France}
  \email{julien.maubon@univ-lorraine.fr}

\date{\today}

\sloppy

\begin{abstract}      
Let $\r$ be a maximal representation of a uniform lattice $\G\subset\SU(n,1)$, $n\geq 2$, in a classical
Lie group of Hermitian type $G$.  We prove that necessarily $G=\SU(p,q)$ with $p\geq qn$ and there exists a
holomorphic or antiholomorphic $\r$-equivariant map from the complex hyperbolic space to the
symmetric space associated to $\SU(p,q)$. This map is moreover a totally geodesic homothetic
embedding. In particular, up to a representation in a compact subgroup of $\SU(p,q)$, the
representation $\r$ extends to a representation of $\SU(n,1)$ in $\SU(p,q)$.      
\end{abstract}
 
\maketitle

\setcounter{tocdepth}{3} 
\tableofcontents

\section{Introduction}

Lattices in non compact simple Lie groups can be regrouped in two broad classes: those which are
superrigid and those which are not. A lattice $\G$ in a
simple noncompact Lie group $H$ is superrigid (over $\R$ or $\C$) if for all simple noncompact
Lie group $G$ with trivial center, every homomorphism $\G\pfd G$ with Zariski-dense image extends to a
homomorphism $H\pfd G$.     
Lattices in simple Lie groups of real rank at least 2, such as ${\rm SL}(n,\mathbb Z)$ in ${\rm
  SL}(n,\R)$ for $n\geq 3$, as well as lattices in the real rank 1 
Lie groups ${\rm Sp}(n,1)$ and ${\rm F}_4^{-20}$, are superrigid
by~\cites{Margulis,CorletteArchimedean,GromovSchoen}. This implies that these lattices are all arithmetic.   
On the other hand, lattices in the remaining simple Lie groups of real rank 1, ${\rm SO}(n,1)$
and $\SU(n,1)$, are not superrigid in general. In particular, the study of their representations
does not reduce to the study of the representations of the Lie group they live in. There are however
important differences between real hyperbolic lattices,  
i.e. lattices in  ${\rm SO}(n,1)$, and complex hyperbolic lattices, i.e. lattices in
$\SU(n,1)$. Real hyperbolic objects are softer and more flexible than their complex
counterparts. From the perspective of representations of lattices, for example, it is sometimes
possible to deform non trivially lattices of ${\rm SO}(n,1)$ in ${\rm SO}(m,1)$, $m>n\geq 3$, see 
e.g.~\cite{JohnsonMillson}. The analogous statement does not
hold for lattices in $\SU(n,1)$, $n\geq 2$: W.~Goldman and
J.~Millson~\cite{GoldmanMillsonLocalRigidity} proved that if $\G\in\SU(n,1)$, $n\geq 2$, is a
uniform lattice and if 
$\r:\G\pfd\SU(m,1)$, $m\geq n$, is the composition of the inclusion $\G\hookrightarrow\SU(n,1)$ with
the natural embedding $\SU(n,1)\hookrightarrow\SU(m,1)$, then $\r$, although not necessarily
infinitesimally rigid, is locally rigid. From a maybe more subjective point of view,   
non arithmetic lattices in $\SO(n,1)$ can be constructed for all $n$~\cite{GPS} but there are no similar constructions
in the complex case and examples of non arithmetic lattices in $\SU(n,1)$ are very difficult to
come by (and none are known for $n\geq 4$).  

\medskip

We will be interested here in global rigidity results for representations of complex hyperbolic
lattices in semisimple Lie groups of Hermitian type with no compact factors which generalize the
local rigidity we just 
mentioned. Recall that a Lie group $G$ is of Hermitian type if its associated
symmetric space is a Hermitian symmetric space. The classical noncompact groups of Hermitian type are
$\SU(p,q)$ with $p\geq q\geq 1$, $\SO_0(p,2)$ with $p\geq 3$, ${\rm Sp}(m,\R)$ with $m\geq 2$ and
$\SO^\star(2m)$ with $m\geq 4$.  

\medskip

Let $\G$ be a lattice in $\SU(n,1)$. The group $\G$ acts on
complex hyperbolic $n$-space $\H=\SU(n,1)/{\rm S}(\U(n)\times\U(1))$. The space $\H$ is the rank 1 Hermitian
symmetric space of non compact type and of complex dimension $n$. From a Riemannian point of view,
it is up to isometry the unique complete simply connected K\"ahler manifold of constant negative holomorphic
sectional curvature. The $\SU(n,1)$-invariant metric on $\H$ will be normalized so that its holomorphic
sectional curvature is $-1$. As a bounded symmetric domain, $\H$ is biholomorphic to the unit ball
in $\C^n$. 

For simplicity in this introduction, and because this is needed in our main result, the lattice $\G$
is assumed to be uniform (and torsion free) unless otherwise specified, so that the quotient
$X:=\G\backslash\H$ is a compact K\"ahler manifold. 

\smallskip     

Let also $G$ be a semisimple Lie group of Hermitian type without compact factors, $\Y$ the
symmetric space associated to $G$ and $\r$ a 
representation of $\G$ in $G$, i.e. a group homomorphism $\r:\G\pfd G$.
There is a natural way to measure the ``complex size'' of the representation $\rho$ by using the
invariant K\"ahler forms of the involved symmetric spaces. The {\em Toledo invariant} of $\r$ is defined       
as follows: 
$$
\tau(\rho)=\frac{1}{n!}\int_X f^\star\o_\Y\wedge\o^{n-1}, 
$$
where $f:\H\pfd\Y$ is any $\rho$-equivariant map, $\o$ is the K\"ahler form of $X$ coming from the invariant
K\"ahler form of $\H$, $\o_\Y$ is the $G$-invariant
K\"ahler form of $\Y$ normalized so that its holomorphic 
sectional curvatures are in $[-1,-1/\rk\Y]$, and $f^\star\o_\Y$ is understood as a 2-form on $X$.    

It should be noted that $\rho$-equivariant maps $\H\pfd\Y$ always exist, because $\Y$ is
contractible, and that any two such maps are equivariantly homotopic, so that the Toledo invariant
depends only on $\r$, not on the choice of $f$. In fact, it depends only on the connected component
of $\Hom(\G,G)$ containing $\r$, because it can be seen as a characteristic class of the flat bundle
on $X$ associated to $\r$. The definition of the Toledo invariant can
be extended to non uniform lattices with a bit more work. 

A fundamental fact about the Toledo invariant that was established in full generality by M.~Burger
and A.~Iozzi in~\cite{BI07} is that it satisfies the following {\em Milnor-Wood type inequality}: 
$$ 
|\tau(\rho)|\leq {\rm rk}(\cal Y)\,{\rm vol} (X).
$$
This allows to single out a special class of representations, namely those for which this inequality
is an equality. These are the {\em maximal representations} we are interested in. 

\medskip

The Toledo invariant was first considered for representations of surface groups, i.e. when $\G$ is
the fundamental group of a closed Riemann surface, which can be seen as a uniform
lattice in $\SU(1,1)$.  It appeared for the first time in D.~Toledo's 1979
paper~\cite{ToledoHarmonic} and more 
explicitly in~\cite{ToledoRepresentations}, where the Milnor-Wood inequality was proved for
$n=1$ and $\rk\,\Y=1$, 
namely when $G=\SU(m,1)$ for some $m\geq 1$. Toledo proved that  
maximal representations are faithful with discrete image, and stabilize a complex line in 
complex hyperbolic $m$-space, thus generalizing a theorem of Goldman for $G={\rm
  SL}(2,\R)$~\cites{GoldmanThesis,GoldmanComponents}. Analogous results in the non uniform case were
proved in~\cites{BI07,KozMobRank1}. L.~Hernandez
showed in~\cite{Hernandez} that maximal representations of surface groups in $G=\SU(p,2)$, $p\geq 2$, are also discrete
and faithful and stabilize a symmetric subspace associated to the subgroup $\SU(2,2)$ in $\Y$. Maximal representations of surface groups are now known to be reductive, discrete and
faithful, to stabilize a maximal tube type subdomain in $\Y$, and in general to carry interesting
geometric structures, see e.g.~\cites{BIW,GW}. They are nevertheless quite flexible. They can for
example always be deformed to representations that are Zariski-dense in the subgroup corresponding
to the tube type subdomain they stabilize \cite{BIW}. 

\medskip

On the other hand, as indicated by the local rigidity result of~\cite{GoldmanMillsonLocalRigidity},
maximal representations of higher dimensional complex hyperbolic lattices, that 
is, lattices in $\SU(n,1)$ for $n$ greater than 1, are expected to be much more rigid. 

This was confirmed for rank 1 targets by K.~Corlette
in~\cite{CorletteFlatGBundles} (the statement was given for representations maximizing the so-called
volume  instead of the Toledo invariant but the proof for the Toledo invariant is essentially the
same). Corlette proved that if $\rho$ is a volume-maximal representation of a uniform lattice
$\G\subset \SU(n,1)$, $n\geq 2$,  in $G=\SU(m,1)$, then there exists a $\rho$-equivariant 
holomorphic totally geodesic embedding $\H\pfd{\mathbb H}_\C^m$.  
This answered a conjecture of Goldman and Millson and implies   
the local rigidity of~\cite{GoldmanMillsonLocalRigidity}. This was later shown to hold also in the case of non uniform
lattices~\cites{BI08,KozMobRank1}.  

\smallskip

For $n\geq 2$ and higher rank targets, the situation was until now far from being
well understood. The case of real rank 2 target Lie groups has been treated in~\cite{KozMobRank2}
(for uniform lattices), but the proof did not go through to higher ranks. In \cite{BIW09},
  M.~Burger, A.~Iozzi and A.~Wienhard proved that maximal representations are necessarily reductive
  (this holds also for $n=1$ and without assuming the lattice to be uniform).  
Very recently, M.~B.~Pozzetti~\cite{Pozzetti} succeeded in generalizing the approach of~\cite{BI08} and
proved that for $n\geq 2$ 
there are no Zariski dense maximal representations of a lattice $\G\subset\SU(n,1)$ in
$\SU(p,q)$ if $p>q>1$. There is no rank restriction in her result, and it is also valid for non
uniform lattices, but as of now it seems to depend strongly on having a non tube type target (this
is the meaning of the assumption $p\neq q$).  

\medskip

In this paper, we prove the expected global rigidity for maximal representations of uniform lattices
of $\SU(n,1)$, $n\geq 2$, in all classical Lie groups of Hermitian type:

\begin{theo}\label{main}
  Let $\G$ be a uniform (torsion free) lattice in $\SU(n,1)$, $n\geq 2$. Let $\rho$ be a group
  homomorphism of $\G$ in a classical noncompact Lie group of Hermitian type $G$, i.e. $G$ is either
  $\SU(p,q)$ with $p\geq q\geq 1$, $\SO_0(p,2)$ with $p\geq 3$,  ${\rm Sp}(m,\R)$ with $m\geq 2$, or
  $\SO^\star(2m)$ with $m\geq 4$.  

If $\r$ is maximal, then $G=\SU(p,q)$ with $p\geq qn$, $\r$ is
  reductive and there exists a holomorphic or antiholomorphic $\r$-equivariant
  map from $\H$ to the symmetric space $\Y_{p,q}$ associated to $\SU(p,q)$.
\end{theo}

As a consequence, maximal representations can be described completely: 

\begin{coro}\label{description}
Let $n\geq 2$ and $p\geq qn$. Let $\r:\G\pfd\SU(p,q)$ be a maximal representation of a uniform
torsion free lattice $\G\subset\SU(n,1)$. Then:

\ndash\, the $\r$-equivariant holomorphic or antiholomorphic map
$\H\pfd\Y_{p,q}$ whose existence is guaranteed by Theorem~\ref{main} is a  totally geodesic
homothetic embedding; it is unique up to composition by an element of $\SU(p,q)$;

\ndash\,  the representation $\r$ is faithful, discrete, and $\r(\G)$
stabilizes (and acts cocompactly on) a totally geodesic image of $\H$ in $\Y_{p,q}$, of induced
holomorphic sectional curvature $-\frac 1q$;

\ndash\, up to 
conjugacy, the representation $\r$ is a product $\r_{\rm diag}\times\r_{\rm cpt}$, where $\r_{\rm
  diag}$ is the standard diagonal embedding $\SU(n,1)\hookrightarrow \SU(n,1)^q
\hookrightarrow \SU(nq,q) 
\hookrightarrow\SU(p,q)$, and $\r_{\rm cpt}$ is a representation of $\G$ in the centralizer
of $\r_{\rm diag}(\SU(n,1))$ in $\SU(p,q)$, which is compact. 
\end{coro}

Because as we said the Toledo invariant is constant on connected components of $\Hom(\G,G)$, this
also implies the local rigidity of maximal representations and in particular we have: 

\begin{coro}
  Let $n\geq 2$ and $p\geq qn$. Then the restriction to a uniform lattice
$\G\subset\SU(n,1)$ of the standard diagonal embedding 
  $\r_{\rm diag}:\,\SU(n,1)\hookrightarrow \SU(n,1)^q 
\hookrightarrow\SU(p,q)$  is locally rigid (up to a representation in the compact
  centralizer of $\r_{\rm diag}(\SU(n,1))$ in $\SU(p,q)$).   
\end{coro} 

This last corollary is in fact true without assuming the lattice $\G$ to be uniform
~\cite{Pozzetti}*{Corollary~1.5}. It is also a special case of the main result
of~\cite{Klingler}, where 
B.~Klingler gave a general algebraic condition for representations 
of uniform lattices in $\SU(n,1)$ induced by representations of $\SU(n,1)$ to be locally rigid.  

\medskip

To prove Theorem~\ref{main}, we work with a reductive representation $\r:\G\pfd G$ (non reductive
representations can be ruled out a priori by~\cite{BIW09}, or later, see \textsection \ref{nonreductive})
and we consider 
the {\em harmonic Higgs bundle} $(E,\t)$ on the closed complex hyperbolic manifold $X=\G\backslash\H$
associated to $\r$ by the work of 
K.~Corlette~\cite{CorletteFlatGBundles} and C.~Simpson~\cite{S2}. This Higgs bundle is {\em polystable}
and has a {\em real structure} which comes from the fact that it  
is constructed out of a representation in a Lie group of Hermitian type (and not
merely in the general linear group). The Toledo invariant is interpreted in this setting as the
degree of a vector bundle on $X$. See~\textsection \ref{higgs}, \textsection \ref{sec:prelim} and
\textsection \ref{sec:ineq}.
These facts can be used in some situations to (re)prove the 
Milnor-Wood inequality and study maximal representations. This has been widely done for
representations of surface groups, see e.g.~\cites{Xia, MarkmanXia, BGPG1, BGPG2}, and also, with
limited success, for higher dimensional lattices \cite{KozMobRank2}. 

The main novelty here is the study of the interplay between the Higgs bundle point of view and the
geometry and dynamics of the {\em tautological foliation} $\T$ on the
projectivized tangent bundle $\PTX$ of the complex hyperbolic manifold $X$. 

When the base (K\"ahler) manifold $Y$ of a harmonic Higgs bundle $(E,\t)\pfd Y$ comes with a smooth
  holomorphic foliation $\T$ by complex curves, and this
  foliation admits 
  an invariant transverse measure, one can investigate the behaviour of the Higgs bundle {\em along the leaves}
  of $\T$.  This is the content of~\textsection \ref{higgsfol}. The transverse measure indeed
  yields a closed current of integration and we define the {\em foliated degree}
  of a coherent   
sheaf on $Y$ by integrating its first Chern class against it. We call a 
subsheaf of $\cal O_Y(E)$ a  {\em leafwise Higgs subsheaf} of $E$ if it is invariant by
the Higgs field $\t$ in the directions tangent to the leaves. With these definitions we 
introduce notions of {\em  leafwise semistability} and {\em leafwise polystability} and we prove
(Proposition~\ref{dec}) that   
they are satisfied by the Higgs bundle $(E,\t)$ when the invariant transverse measure is induced by
an invariant
transverse volume form.

Now, there is a well-defined 
notion of {\em complex geodesics} in complex hyperbolic space $\H$. This implies that the
projectivized tangent bundle $\PTX$ of the complex hyperbolic manifold $X=\G\backslash\H$ carries 
a smooth holomorphic $1$-dimensional foliation $\T$ by lifts of tangent spaces of (local) complex
geodesics, see~\textsection \ref{cplxgeod}.  The tangential line subbundle $L$ of the tangent 
bundle of $\PTX$, i.e. the subbundle of tangent vectors tangent to 
the leaves of the foliation, identifies naturally  with the tautological line bundle $\OX$ on $\PTX$. 
The tautological foliation is endowed with a homogeneous transverse structure, where the 
$\SU(n,1)$-homogeneous space in question is the space $\cal G$ of complex geodesics of $\H$. This space
supports an invariant indefinite but non degenerate K\"ahler metric $\o_{\cal G}$, hence an
invariant volume form which defines a transverse measure $\mu_{\cal G}$ for the foliation $\T$,
cf.~\textsection \ref{sec:transverse}. The fundamental feature of the induced current of integration
  is that it enables to compute the Toledo invariant of the representation $\r$ and degrees of vector bundles
  on $X$ as foliated degrees of vector bundles on $\PTX$ (Proposition~\ref{prop:transverse}).

The idea is then to pull-back the Higgs bundle $(E,\t)\pfd X$ associated to the
  representation $\r$ to obtain a harmonic Higgs bundle $(\tilde E,\tilde\t)$ over the projectivized tangent bundle
$\PTX$ and to take advantage of the leafwise stability properties of this new Higgs bundle with
respect to the tautological foliation $\T$ and its invariant transverse measure $\mu_{\cal G}$. This 
allows to give a new proof of the   
Milnor-Wood inequality for reductive representations of uniform lattices and to gain a lot of
information in the maximal case, see e.g.~\textsection \ref{sec:reduc} for representations in
  $\SU(p,q)$. To conclude the proof one needs a 
dynamical argument to understand closures of projections to $X$ of subsets of $\PTX$ which are
saturated under the tautological foliation. This is done 
using results of M.~Ratner on unipotent flows, see~\textsection \ref{sec:closure}.   

\medskip

The interpretation of the Toledo invariant as a ``foliated Toledo number'' is 
sketched by M.~Burger and A.~Iozzi in~\cite{BI08}*{p.~183}, where it is attributed to
F.~Labourie. This point of view is indeed strongly related with their approach, and the one of
M.~B.~Pozzetti, where one wants to prove that when a representation is 
maximal, there exists an 
equivariant measurable map between the Shilov boundaries that preserves a special incidence 
geometry. In the complex hyperbolic case, this incidence geometry is the  
geometry of {\em chains}, i.e. of boundaries at infinity of complex geodesics. Tautological
foliations on the projectivized tangent bundle of manifolds carrying a holomorphic projective
structure (in particular complex hyperbolic manifolds) are also discussed and used by N.~Mok
in~\cite{Mok}.  Some time ago, without at first grasping the foliated 
side of the story, the authors of the present paper made some quickly unsuccessful attempts
at working with Higgs bundles on the projectivized tangent bundle. Reading
F.~Labourie's suggestion in~\cite{BI08} and N.~Mok's article~\cite{Mok} 
encouraged them to try again. 

Combining foliations and Higgs bundle theory to prove rigidity properties of lattices is of course
  reminiscent of the work of M. Gromov on foliated harmonic maps
  in~\cites{Gromov1,Gromov2}. M. Gromov considered foliations by (lifts of) totally geodesic
  subspaces in (bundles over) locally 
  symmetric spaces of which the tautological foliation discussed here is a particular case. There is
  however a 
  difference. In his application to quaternionic rigidity \cite{Gromov2}*{\textsection 7.E} (see
  also~\cite{CorletteArchimedean} 
  for a different proof), M. Gromov uses his existence theorem for foliated harmonic maps to produce maps on
  quaternionic hyperbolic space which 
  are harmonic along totally geodesic complex subspaces but not (a priori)
  harmonic on the whole space, because the harmonic map on the whole space is not (a priori)
  harmonic when restricted to these subspaces. In this paper, since we work on K\"ahler manifolds, and
  our leaves are complex curves, harmonic maps are pluriharmonic (see \textsection
  \ref{higgs}) and their restrictions to the leaves are automatically harmonic. Therefore foliated
  harmonic maps are not needed and neither is a fully fledged theory of foliated Higgs bundles
  (e.g. on real manifolds foliated by K\"ahler submanifolds), although such a theory would probably be
  interesting to develop.

\bigskip

{\em Acknowledgements.} We are very grateful to Matei Toma for the time he accepted to spend
discussing various topics, and in particular for his help concerning complex
analytic aspects of foliations. We thank Jean-François Quint, Benoît Claudon and Fr\'ed\'eric Touzet
for useful conversations around the subject of this paper. We would also like to thank
  the referees for their valuable comments and suggestions which aided a lot in
  improving the quality and readability of the paper. 

\section{Higgs bundles on foliated K\"ahler manifolds}\label{higgsfolkaehler}

In this section, we first give a brief account on harmonic Higgs bundles on a compact K\"ahler
  manifold $Y$. When the manifold $Y$ admits a holomorphic foliation by complex curves and the
  foliation admits 
    an invariant transverse 
  measure, we define a notion of foliated degree for $\cal O_Y$-coherent sheaves. If moreover the
  transverse measure is induced by
  an invariant transverse volume form, we exhibit some stability properties of the Higgs bundle with
  respect to the foliated degree.

\subsection{Harmonic Higgs bundles}\label{higgs}\hfill

\smallskip

Let $\mf$ be a compact manifold, $\G$ its fundamental group, and $\r:\G\pfd G$ a group
  homomorphism in a real algebraic semisimple Lie group without compact factors
  $G\subset \SL(N,\C)$. 

We assume in this section that $\r$ is {\em reductive}, i.e. 
the Zariski closure of  $\rho(\G)$ in $G$ is a reductive group. By a fundamental result of 
K.~Corlette~\cite{CorletteFlatGBundles}, this is equivalent to the existence of a
$\r$-equivariant {\em harmonic map} $f$ from the universal cover $\tilde \mf$ of $\mf$ to the
symmetric space $\Y$ associated to $G$.

When the manifold $\mf$ is moreover K\"ahler, it follows from a Bochner formula
due to J.~H.~Sampson~\cite{Sampson} and Y.-T.~Siu~\cite{Siu}, that the harmonic map $f$ is 
pluriharmonic, and that the image of the $(1,0)$-part $d^{1,0}f:T^{1,0}\tilde \mf\pfd T^\C\Y$
of its complexified differential is Abelian (as a subspace of the
complexification of the Lie algebra of $G$). 
 This has been shown by C.~Simpson~\cites{S1,S2} to give a {\em harmonic Higgs bundle} $(E,\t)$ on
 $\mf$. 

 The fact that the representation $\r$ takes its values in the real group $G$ endows the Higgs
  bundle with a {\em real 
  structure}. This real structure is important and will be discussed later in different particular
cases. However, at this point it is not relevant and in this section we see $\r$ as a homomorphism 
in $\SL(N,\C)$. 

\medskip

The bundle $E$, as a $C^\infty$-bundle, is the flat complex vector bundle of rank $N$ with holonomy $\rho$. The Higgs field $\t$ is a holomorphic $(1,0)$-form
with values in $\End(E)$, which can be seen as the $(1,0)$-part $d^{1,0}f$ 
of the complexified differential of the harmonic map $f$. It satisfies the {\em integrability
  condition} $[\t,\t]=0$. The
$\r$-equivariant harmonic map
can also be thought of as a reduction of the structure group of $E$ to the maximal compact
  subgroup $\SU(N)$ of $\SL(N,\C)$, see e.g.~\cite{CorletteFlatGBundles}. Therefore choosing a
  $\SU(N)$-invariant Hermitian metric on a 
  fiber of $E$ defines 
a Hermitian metric on $E$, called the {\em harmonic metric}, which has the following properties. If $D$ is
the flat connection on $E$ and $\nabla$ the component of $D$ which preserves this metric, then 
$(\nabla'')^2=0$ and $\nabla''\t=0$, so that $\nabla''$ defines a holomorphic structure on $E$ for
which $\t$ is holomorphic. Moreover $D$, $\nabla$ and $\t$ are related by
$$
D=\nabla+\t+\t^\star,
$$         
where $\t^\star$ is the adjoint of $\t$ w.r.t. the harmonic metric. This, together with the
Chern-Weil formula, implies that $(E,\t)$ is a {\em polystable} Higgs bundle of degree 0 on $\mf$,
see~\cite{S1} 
and the proof of Proposition~\ref{dec} below. This means first that $(E,\t)$ is a {\em
  semistable} Higgs bundle, namely that if ${\cal F}\subset {\cal O}_\mf(E)$ is a Higgs subsheaf of
$E$, i.e. a subsheaf such that $\theta({\cal F\otimes T_\mf})\subset {\cal F}$, 
then 
$$
\deg {\cal F}:=\frac{1}{m!}\,\int_\mf c_1(\cal F)\wedge\o_\mf^{m-1}\,\leq\,0\,=\,\deg E
$$
where $m$ is the dimension and $\o_\mf$ the K\"ahler form  of $\mf$ (the last equality holds because
$E$ is flat). 
Second, whenever $\cal F$ is a Higgs subsheaf of $E$ of degree equal to $0$, its
  saturation (see below) is the sheaf of 
sections of a holomorphic vector subbundle $F$ of $E$ stable by $\t$ and the orthogonal complement $F^\perp$
of $F$ w.r.t the harmonic metric is also a holomorphic subbundle of $E$ stable by $\t$, so that we
have a Higgs bundle orthogonal decomposition 
$$  
(E,\t)= (F,\t_{|F})\oplus (F^\perp,\t_{|F^\perp}).
$$

\begin{rema}
  In general, when dealing with notions of stability of vector bundles, one uses
  slopes rather than degrees. However, in our case $E$ is flat and there is no need to consider
  slopes.
\end{rema}

\subsection{Higgs bundles and foliations}\label{higgsfol}\hfill

\smallskip

Assume now that the compact K\"ahler manifold $\mf$, with its harmonic Higgs bundle $(E,\t)$, also
admits a smooth holomorphic foliation $\T$ by complex curves, and that this
foliation has an invariant transverse invariant (positive) measure $\mu$.  
Our goal in this section is to understand the behaviour of the
Higgs bundle $(E,\t)$ with respect to the foliation $\T$ and its transverse measure $\mu$.

\medskip

We begin by defining adapted notions of Higgs subsheaves and degree.

\smallskip

Let $L\subset T_\mf$ be the tangential line field of the foliation $\T$, i.e. the holomorphic line subbundle of vectors which are tangent to the leaves of
$\T$, and let $L^\vee$ be its dual. We restrict the Higgs field $\t$ to $L$, i.e. we see it as a holomorphic
section of $\End(E)\otimes L^\vee$. 

A subsheaf ${\cal F}\subset{\cal O_\mf}(E)$ is {\em invariant along the leaves} or is a {\em
  leafwise Higgs subsheaf} if the Higgs field $\t$ maps $\cal F\otimes L$ to $\cal F$.  

\smallskip

We make
  the observation that a Higgs subsheaf of $(E,\t)$ is a leafwise Higgs subsheaf, but that
  the converse does not hold, so that there is no reason why the degree (computed w.r.t. the K\"ahler
  form $\o_Y$) of a leafwise Higgs subsheaf should be nonpositive.

\smallskip

The invariant transverse measure $\mu$ defines a closed current $\intO$ of bidegree $(m-1,m-1)$ on $\mf$ which is
$\T$-invariant and positive since $\mu$ is (see~\cite{God}*{V.3.5}
and~\cite{Sullivan}). Let indeed $\a$ be a $2$-form on $\mf$. Take a covering $(U_i)_{i\in I}$ of
$\mf$ by regular open sets  
for the foliation $\T$, and a partition of unity $(\chi_i)_{i\in I}$ subordinated to it. Let $T_i$
be the space of plaques of $U_i$ and call again $\mu$ the measure on $T_i$ given by the transverse
measure. The forms $\chi_i\a$ are compactly supported in the open sets $U_i$
and by integrating $\chi_i\a$ on the plaques of $U_i$, we obtain a compactly supported function on
the space $T_i$  which we can then integrate against the measure $\mu$ to get
$$
\intO\a:=\sum_{i\in I}\int_{T_i}\left(\int_t \chi_i\, \a\right)d\mu(t).
$$ 

Invariant transverse measures to $\T$ can be atomic, for example 
when $\T$ admits a closed leaf $C$ in which case the current is given by integration on $C$. Or they
can be diffuse, for example when the foliation admits an invariant transverse volume form
(cf.~\cite{God}*{V.3.7(i)} for the definition) in which case  
there exists a smooth closed {\em basic} $(m-1,m-1)$-form $\Omega$ on $\mf$ such that
$\intO\a=\int_\mf\a\wedge\Omega$ for any smooth $2$-form $\a$ on $\mf$. A form $\Omega$ is basic
w.r.t the foliation $\T$ if $\iota_\xi\Omega=\iota_\xi d\Omega=0$ for all $\xi\in L$.

The {\it foliated degree} $\degO {\cal F}$ of an ${\cal O}_\mf$-coherent sheaf ${\cal
  F}$ on $\mf$ is defined by $$\degO {\cal F}=\intO c_1({\cal F})$$ where
$c_1({\cal F})$ is any smooth representative of the first Chern class of ${\cal F}$. 

\medskip

Our main technical tool will be a {\em weak polystability} property of the harmonic Higgs bundle
{\em along the leaves}, in the case when the current $\intO$ is sufficiently regular, namely when
the transverse measure comes from a  transverse volume form.
Before giving the statement,  
we need some definitions (unfortunately, because this seems to be the admitted terminology in the
literature, we have to use the word ``saturated'' with two different meanings, but no confusion
should arise). 

A subset $S\subset \mf$ is {\em
  $\T$-saturated} if it is a union of leaves of the foliation $\T$, i.e. for all $x\in S$, the
leaf $\L_x$ of the foliation $\T$ through $x$ is 
included in $S$. If $S$ is $\T$-saturated then so is $\mf\backslash S$.  

A coherent subsheaf $\cal F$ of the sheaf $\cal O_Y(E)$ is {\em saturated} if $\cal
  O_Y(E)/\cal F$ is torsion free. A saturated subsheaf of $\cal O_Y(E)$ is reflexive and therefore
  normal. If $\cal F$ is a
  coherent subsheaf of $\cal O_Y(E)$, its {\em saturation} is the kernel of $\cal O_Y(E)\pfd (\cal O_Y(E)/\cal F)/{\rm
    Tor}(\cal O_Y(E)/\cal F)$. It is a saturated subsheaf of $\cal O_Y(E)$.  

In this paper, the {\em singular locus}  $\S(\cal F)$  of a
  coherent subsheaf $\cal F$ of $\cal O_Y(E)$ is the analytic subset of $Y$ where 
the quotient $\cal O_Y(E)/\cal F$ is not locally free. This is not the usual definition. The
  complement $Y\backslash\S(\cal  F)$ of $\S(\F)$ is 
the biggest subset of $Y$ where $\cal F$ is the sheaf of sections of a subbundle $F$ of $E$. If
$\cal F$ is saturated then $\S(\cal F)$ has codimension at least 2 in $Y$.

\smallskip

\begin{prop}~\label{dec} 
 Let $\mf$ be a compact K\"ahler manifold and $(E,\t)$ be a harmonic Higgs
    bundle on $\mf$. Let $\T$ be a smooth holomorphic foliation of $\mf$ by complex curves.
    Assume that $\T$ admits an invariant transverse measure $\mu$ given by an invariant
      transverse volume form. 
\begin{enumerate}
\item\label{semistab} {\em (Semistability along the leaves)} For any leafwise Higgs subsheaf ${\cal
    F}\subset{\cal O}_{\mf}(E)$ of $(E,\t)$,  $\degO{\cal F}\leq 0$.
\item\label{polystab} {\em (Weak polystability along the leaves)} If ${\cal F}\subset{\cal O}_{\mf}(E)$ is a
  saturated leafwise Higgs subsheaf of $(E, \t)$ such that $\degO{\cal F}= 0$, then
\begin{enumerate} 
\item\label{sat} the singular locus $\sing$ of $\F$ is $\T$-saturated;  
\item\label{split} on $Y\backslash\sing$, if $F$ is the subbundle of $E$ such that $\F=\cal O_Y(F)$
  and $F^\perp$ is its orthogonal complement w.r.t. the harmonic metric on $E$, we have
$\t(F^\perp\otimes L)\subset F^\perp$ and the
$C^\infty$-decomposition $E=F\oplus F^\perp$ is holomorphic along the leaves of the foliation
$\T$, i.e. for any leaf $\cal L$ of $\T$ such that $\cal L\subset \mf\backslash \sing$, $E_{|\cal
  L}=F_{|\cal L}\oplus F^\perp_{|\cal L}$ is a holomorphic orthogonal direct sum on $\cal L$.
\end{enumerate}
\end{enumerate}
\end{prop}

\begin{proof} 
We first prove the semistability along the leaves, using the Chern-Weil formula. Let $\Omega$ be the
closed basic $(m-1,m-1)$-form on $\mf$ given by the invariant transverse 
    volume form to the foliation $\T$, so that for all $2$-form $\a$ on $\mf$,
    $\intO\a=\int_\mf\a\wedge\Omega$.   

Let $\F$ be a leafwise Higgs subsheaf of $(E,\t)$. 
The foliated degree
of the saturation $\overline \F$ of $\F$ is greater than or equal
to the foliated degree of $\F$. This is because there exists an effective divisor $D$ such that
  $\det\overline\F=(\det\F)\otimes [D]$ (see e.g.~\cite{K}*{Chap.~V (8.5) p. 180}),  so that
  $\degO\overline\F=\intO c_1(\ov\F)=\int_\mf c_1(\det \ov
  F)\wedge\Omega=\degO\F+\int_D\Omega\geq\degO\F$. Therefore it is enough to prove~(\ref{semistab})
  for saturated leafwise Higgs subsheaves and we assume for now on that $\F$ is
  saturated, so that the codimension of $\sing$ is at least 2. 

There exists a
holomorphic subbundle 
 $F$ of $E$ defined outside of the singular locus $\sing$ of $\F$, such that $\cal F$ is the sheaf
 of sections of $F$  on $\mf\backslash\sing$. On $\mf\backslash \sing$, we can decompose the
flat connection $D=\nabla+\t+\t^\star$ with  
respect to the orthogonal decomposition $E={F}\oplus{F}^\perp$ (for the harmonic metric). Denoting by $\sigma\in
C^\infty_{1,0}(\mf\backslash\sing,{\rm Hom}({F},{F}^\perp))$ the second fundamental form of ${F}$, we get: 
$$ 
 D=
\left(
\begin{matrix}
\nabla_{F} & -\sigma^\star\\
\sigma & \nabla_{{F}^\perp}\\
  \end{matrix}
\right)+
\left(
\begin{matrix} 
\theta_1 & \theta_2 \\
  \theta_3 & \theta_4\\
 \end{matrix}
\right) 
+
\left(
\begin{matrix} 
\theta_1^\star & \theta_3^\star \\
  \theta_2^\star & \theta_4^\star\\
 \end{matrix}
\right)
 .
$$

On the one hand, the curvature ${\Theta}_{F}$ of the connection
$\nabla_{F}+\theta_1+\theta_1^\star$ can be used to compute a 
representative of the first Chern class of ${\cal F}$ on $\mf\backslash \sing$, namely
$c_1({F})=\frac {\sqrt{-1}}{2\pi}{\rm tr}\,{\Theta}_{F}$, 
and integrating $\frac{\rmo}{2\pi}\, {\rm tr}\,{\Theta}_{F}\wedge\Omega$ on $\mf\backslash
\sing$ gives the foliated degree of $\cal F$. This is sketched
in~\cite{S1}*{Lemma~3.2}, and can be proved as follows. By \cite{K}*{pp.~180-182} (see
also~\cite{Sibley}*{Theorem~2.23 \& Lemma~4.6}), if $\Xi_{F}$ is the curvature of the metric
connection $\nabla_F$ on $Y\backslash \sing$ then integrating $\frac {\sqrt{-1}}{2\pi}{\rm
  tr}\,{\Xi}_{F}$ against $\Omega$ on $Y\backslash \sing$ computes $\int_Y c_1({\cal
  F})\wedge\Omega$. Moreover, still on $Y\backslash \sing$, we have ${\rm tr}\,{\Theta}_{F}={\rm
  tr}\,{\Xi}_{F}+d \bigl({\rm tr}\,(\theta_1+\theta_1^\star)\bigr)$. It is known
(see~\cite{UhlenbeckYau} and also~\cite{Pop}) that the orthogonal projection $\varpi:E\rightarrow
F$,  
which can be seen as an element of $L^\infty(Y,\End(E))$, is also in the Sobolev space
$L_1^2(Y,\End(E))$. Therefore 
$\theta_1=\varpi\circ\theta\circ\varpi$ and 
$\theta_1^\star=\varpi\circ\theta^\star\circ\varpi$ are such that $\int_{Y\backslash \sing} d\bigl({\rm
  tr}\,(\theta_1+\theta_1^\star)\bigr)\wedge\Omega=\int_{Y} d\bigl({\rm
  tr}\,(\theta_1+\theta_1^\star)\bigr)\wedge\Omega=0$ by Stokes formula and density of smooth
functions in $L^2(Y)$. To sum up, we have
$$
\degO\F = \int_Y c_1(\F)\wedge\Omega=\frac{\sqrt{-1}}{2\pi}\int_{Y\backslash\sing}\tr
\Xi_F\wedge\Omega = \frac{\sqrt{-1}}{2\pi}\int_{Y\backslash\sing}\tr
\Theta_F\wedge\Omega.
$$

On the other hand, since $ D^2=0$, we have
$(\nabla_{F}+\theta_1+\theta_1^\star)^2=-(\theta_2+\theta_3^\star- 
\sigma^\star)\wedge(\theta_3+\theta_2^\star+\sigma)$. 
Therefore 
$$
\begin{array}{rcl}
\ds {\Theta}_{F}\wedge\Omega & = & 
\ds (-\theta _2\wedge\theta_2^\star
+\sigma^\star\wedge\sigma -\theta_3^\star\wedge\theta_3
-\t_3^\star\wedge\sigma+\sigma^\star\wedge\t_3)\wedge\Omega \\
& = & \ds(-\theta _2\wedge\theta_2^\star
+\sigma^\star\wedge\sigma)\wedge\Omega
\end{array}
$$
because $\t_3\wedge\Omega=\t_3^\star\wedge\Omega=0$, for $\Omega$ is a basic $(m-1,m-1)$-form and
$\t_3$ vanishes in the direction of the leaves of $\T$ since $\F$ is a leafwise
Higgs subsheaf. Hence
$$
\degO\F =\frac{\sqrt{-1}}{2\pi}\int_{Y\backslash\sing}{\rm tr}\,(-\theta _2\wedge\theta_2^\star
+\sigma^\star\wedge\sigma)\wedge\Omega\, \leq\, 0
$$
as wanted.

\medskip

Now let us prove (\ref{polystab}). We will follow the proof that Einstein-Hermitian vector bundles are
polystable, see~\cite{K}. Assume that $\degO(\cal
F)=0$ for the subsheaf $\cal F$ of the proof of Assertion (\ref{semistab}). Then ${\rm tr}\,(-\theta
_2\wedge\theta_2^\star  
+\sigma^\star\wedge\sigma)\wedge\Omega=0$ on $\mf\backslash\sing$ and this implies that for all
$\eta\in {L}_{|\mf\backslash\sing}$, $\t_2(\eta)=0$ and $\sigma(\eta)=0$. This means on the
first hand that
$\t(F^\perp\otimes L)\subset F^\perp$  on $Y\backslash \sing$ and on the second hand that if $\cal L$ is a leaf of $\T$, then on $\cal
L\cap(\mf\backslash\sing)$, $F$ is a parallel subbundle of $E$. As in the proof
of~\cite{K}*{Theorem~5.8.3}, we deduce that the $C^\infty$-decomposition $E=F\oplus F^\perp$ is
holomorphic when restricted to $\cal L\cap(\mf\backslash\sing)$. We will say that the
decomposition, where it is defined, is holomorphic along the leaves of $\T$. This will 
prove~(\ref{split}) once (\ref{sat}) will be established. 

\smallskip

Let $\singsat=\{x\in\sing\mbox{ such that }\cal L_x\subset\sing\}$. This subset of
  $\sing$ is  
$\T$-saturated, and it is an analytic subset of
codimension at least 2 in $\mf$. Indeed, on a regular open set $U$  for the foliation
$\T$ identified with an open subset of $\C^{m}$, we may assume that the leaves of $\T$ are the
fibers of a linear projection 
$p:\C^{m}\pfd\C^{m-1}$. Then $\singsat=\{x\in\sing\mbox{ such that }\dim p^{-1}(p(x))\geq 1\}$
and hence is analytic by~\cite{Fischer}*{p.~137}.  

We will prove that the holomorphic subbundle $F$ which is defined
outside $\sing$ can be extended to a holomorphic subbundle defined on $\mf\backslash\singsat$, and
that the decomposition $E=F\oplus F^\perp$, which is $C^\infty$ and holomorphic along the
leaves outside $\sing$, can also be extended to a decomposition on $\mf\backslash\singsat$, with
the same regularity. This will be a consequence of the following variation on the second Riemann extension
theorem:      
 
\begin{lemma}\label{FunctionsExtend}
Let $\cal O$ be an open subset of $\C^m$ and $V$ be a 1-dimensional linear subspace in $\C^m$. For $z\in\C^m$, let $\ell_z$ be the affine
line $z+V$. Let $A$ be an 
analytic subset of $\cal O$, of codimension at least 2. Let 
  $\varphi:\cal O\backslash A\pfd\C$ be a $C^\infty$ map. Assume that $\varphi$ is holomorphic in the
  $V$-direction, meaning that for every $z\in\cal O\backslash A$, the restriction of $\varphi$ to a
  neighborhood of 
  $z$ in $\ell_z$ is holomorphic. Let $a$ be a point of $A$ which is an isolated point of
  $A\cap\ell_a$. Then there exist a neighborhood $\cal U$ of $a$ in $\cal O$ and a $C^\infty$ map
  $\Phi:\cal U\pfd \C$, holomorphic in the $V$-direction, such that $\Phi=\varphi$ on $\cal
  U\backslash A$.          
\end{lemma}

(We postpone the proof of the lemma to the end of the present proof.)

Let $x$ be a point of $\sing\backslash\singsat$, i.e. $x$ is an isolated
point of $\sing\cap\cal L_x$. We want to show that $F$ and $F^\perp$ can be extended in a neighborhood of $x$ in
$\mf$. Since this is a local problem, we may assume that we are on an open subset $\cal O$
  of $\C^{m}$, that the leaves of the tautological
foliation $\T$ are the affine lines of a given direction $V\subset\C^{m}$ as in the lemma, and that
$E$ is a trivial bundle. Because of the regularity properties of $F$ and
$F^\perp$, the section $\phi$ of $\Hom(E,E)$ defined over $\cal O\backslash\sing$ and
corresponding to the  
orthogonal projection on $F^\perp$ is given by a matrix of functions $(\phi_{ij})$ from
$\cal O\backslash\sing$ to $\C$ which are $C^\infty$ and holomorphic in the $V$-direction. By
the above lemma, $\phi$ extends to a section of $\Hom(E,E)$ defined in a neighborhood of
$x$ in $\cal O$. By the lower semi-continuity of the rank, if $x\in\sing\backslash\singsat$,
$\rk\, \phi(x)\leq \rk\, E-\rk\, F$. In the same way, $\id-\phi$ can be extended to
$\sing\backslash\singsat$ and hence $\rk\,\phi(x)= \rk\,E-\rk\, F$ on
$\sing\backslash\singsat$. Hence the subbundles $F$ and 
$F^\perp$ can be 
extended to $\mf\backslash\singsat$, as $C^\infty$-vector bundles holomorphic along the leaves of
$\cal T$. Since $F$ is holomorphic and orthogonal to $F^\perp$ on $\mf\backslash\sing$, this is
also true on $\mf\backslash\singsat$. Finally, because $\F$ is normal, $\F$ coincides with
  the sheaf of sections of $F$ on $Y\backslash\singsat$.

By the definition of $\sing$, this implies that
    $(Y\backslash\singsat)\cap\sing=\emptyset$, so that $\sing=\singsat$ and hence $\sing$ is
  $\T$-saturated.
\end{proof}

\begin{proof}[Proof of Lemma~\ref{FunctionsExtend}]  

Choose coordinates $(z_1,\ldots,z_m)$ on $\C^m$ such that $a=0$ and
$\ell_0=V=\{z\,|\,z_1=\cdots=z_{m-1}=0\}$. By 
assumption $0$ is an isolated point of $\ell_0\cap A$, hence there exists $r>0$ such that the circle 
$\{z\,|\,z_1=\cdots=z_{m-1}=0,\,|z_m|=r\}$ does not meet $A$. Let $\varepsilon >0$ be
such that the polydisc $\Delta(0,\e)^{m-1}\times\Delta(0,r+\varepsilon)\subset\cal O$ and $\{z \mid |z_i|<\varepsilon,\,\forall 1\leq i\leq m-1,\,|z_m|\in(r-\e,r+\e) \}\cap A=\emptyset$. 
Then the function $\Phi$ defined by
$$
\Phi(z_1,\ldots,z_m)=\frac1{2\pi\rmo}\int_{|t|=r}\frac{\varphi (z_1,\dots,z_{m-1},t)}{t-z_m}\,dt
$$
is $C^\infty$ on the polydisc $\cal U=\Delta(0,\e)^{m-1}\times\Delta(0,r)\subset\cal O$. Moreover,
for all $(z_1\ldots,z_{m-1})\in\Delta(0,\e)^{m-1}$, the map $z_m\mapsto 
\Phi(z_1,\dots,z_{m-1},z_m)$ is holomorphic on the disc $\Delta(0,r)$. 

Let $\cal U'=\{z\in \cal U\mid \ell_z\cap A\cap \cal U=\emptyset\}$. Because $\varphi$ is holomorphic in
the $V$-direction, for all
$z$ in $\cal U '$, the restriction 
of $\Phi$ to $\ell_z\cap \cal U$ equals $\varphi$ by the Cauchy formula. Now $\cal U '$ is dense in
$\cal U$. Indeed, let  $p$ be the projection 
$\C^m\pfd\C^m/V$. Since  $0\in A$ is an isolated point of $A\cap p^{-1}(p(0))$, near $0=p(0)$ the
set $p(A)$ is analytic of the same dimension as $A$ (\cite{Fischer}*{p.~133}), thus it has
codimension at least 1. Hence $\Phi=\varphi$ on $\cal U\backslash A$.  
\end{proof}

\begin{rema}\label{closedleaves1}
{\em Closed leaves I: semistability and saturation of sheaves.}  If instead of an invariant transverse
  measure given by an invariant transverse volume form, one considers the measure $\delta_C$ given by a 
closed leaf $C$ of the foliation $\T$ (assuming there is one), a statement like
Proposition~\ref{dec} will fail without further assumptions. In fact, it is well known that even the
notion of degree is in general not 
reasonable in this case. Suppose for example that $Y$ is a compact K\"ahler
surface and that the foliation $\T$ admits a closed leaf of negative self
intersection. Then $\deg_{\T\hspace{-1pt},\,\delta_C}{\cal O}_\mf(-C):=\int_C{\cal O}_\mf(-C)=-C^2>0$, and the
``degree'' of  ${\cal O}_\mf(-C)$ is bigger than the ``degree'' of its saturation ${\cal O}_\mf$,
which of course vanishes.  In
order to avoid this kind of inconvenience, it is necessary that $C$ enjoys some positivity
properties, e.g. the cohomology class of the current $\int_{\T\hspace{-1pt},\,\delta_C}$ is represented by a smooth
semi-positive $(1,1)$-form in the sense of currents. One can then hope to get a leafwise
semistability result.
\end{rema}

\section{The tautological foliation on the projectivized tangent bundle of
  complex hyperbolic manifolds}\label{tautofol}

In this section we give a detailed description of the tautological foliation $\T$ by complex curves
  on the projectivized tangent bundle $\PTX$ of a complex hyperbolic manifold $X$ and of its transverse
  structure. Together with the results of \textsection\ref{higgsfolkaehler}, it will be one of the main
  tools to (re)prove the Milnor-Wood inequality on the Toledo invariant and to study maximal
    representations. The section ends with some
  applications of Ratner's theorem on the closure of orbits under groups generated by unipotent
  elements to projection to $X$ of subsets of $\PTX$ saturated under $\T$. 

\medskip

The Klein model of complex hyperbolic $n$-space $\H$ is the set of negative lines in
$\C^{n+1}$ for a Hermitian form 
$h$ of signature $(n,1)$. It is an open set in the projective space $\C{\mathbb P}^n$. 

The Lie group $\SU(n,1)=\SU(\C^{n+1},h)$ is the subgroup of $\SL(n+1,\C)$ consisting of elements
preserving  the 
Hermitian form $h$. As a group of matrices, in a basis $(e_1,\ldots,e_n,e_{n+1})$ of $\C^{n+1}$
where the matrix of $h$ is the diagonal matrix $I_{n,1}={\rm diag}(1,\ldots,1,-1)$, 
$$
\SU(n,1)=\{M\in\SL(n+1,\C)\,|\,M^\star I_{n,1}M=I_{n,1}\},
$$
where $M^\star$ denotes the conjugate transpose of $M$. 

The group $\SU(n,1)$ acts transitively on $\H$. The stabilizer of a point is a maximal compact
subgroup of $\SU(n,1)$ and is conjugated to $\U(n)\simeq{\rm S}(\U(n)\times\U(1))$. This gives a
realization of $\H$  as the Hermitian symmetric
space $\SU(n,1)/\U(n)$. As a bounded symmetric domain,  
complex hyperbolic $n$-space is
biholomorphic to the unit ball in $\C^n$.

We equip the Lie algebra ${\mathfrak{su}}(n,1)$ of $\SU(n,1)$ with the Killing form $b(A,B)=
2\,\tr(AB)$, normalized so that the holomorphic sectional curvature of the $\SU(n,1)$-invariant
K\"ahler metric 
$\o$ it induces on $\H$ is $-1$. 

\smallskip

An $n$-dimensional complex hyperbolic manifold $X$ is the quotient of $\H$ by a discrete torsion free
subgroup $\G$ of $\SU(n,1)$.

\subsection{Complex geodesics and the tautological foliation}\label{cplxgeod}\hfill

\smallskip

The {\em complex geodesics} of
$\H\subset\C{\mathbb P}^n$ are the intersections of 
$\H$ with the complex lines $\C{\mathbb P}^1\subset \C{\mathbb P}^n$. It follows that the space $\cal G$ of complex geodesics is an open homogeneous set in the Grassmannian of
2-planes in $\C^{n+1}$. More precisely, $\SU(n,1)$ acts transitively on $\cal G$ and $\cal
G=\SU(n,1)/{\rm S}(\U(n-1)\times\U(1,1))$.  Complex geodesics are complex totally
geodesic subspaces of $\H$ isometric (up to a
constant) to the Poincar\'e disc, of induced sectional curvature $-1$. Given a point in $\H$ and
a complex tangent line at this point, there is a unique complex geodesic through that point tangent
to the complex line.  

\medskip

Let $\TH\pfd \H$ be the holomorphic tangent bundle of $\H$ and consider the {\em projectivized
tangent bundle} $\pi:\PTH\pfd\H$ of $\H$. It is a holomorphic 
bundle and the fiber over a point $x\in \H$ is
the projective space of lines in the tangent space $T_{\H,x}$.  
A point in the projectivized tangent bundle $\PTH$ of 
$\H$ is given by two $h$-orthogonal complex lines in $\C^{n+1}$ spanning a complex geodesic. Hence
$\PTH$ is the homogeneous space $\SU(n,1)/{\rm S}(\U(n-1)\times\U(1)\times\U(1))$. 
The central fiber of the holomorphic
projection $\pi:\PTH\pfd\H$ is 
$\U(n)/(\U(n-1)\times \U(1))=\C{\mathbb P}^{n-1}$ as it should. 

\medskip

The map from $\PTH$ to
$\cal G$ associating to a point in the projectivized tangent 
bundle the complex geodesic it defines is the $\SU(n,1)$-equivariant holomorphic fibration  
$\pig:  \SU(n,1)/{\rm S}(\U(n-1)\times\U(1)\times\U(1))\pfd \SU(n,1)/{\rm
  S}(\U(n-1)\times\U(1,1))$.  
The central fiber $\U(1,1)/(\U(1)\times\U(1))$ is isometric to the Poincar\'e disc so that $\PTH$ is a
disc bundle over $\cal G$.
This of course defines a foliation on $\PTH$ whose leaves are the fibers of $\pig$.  

\medskip

If $\G$ is a discrete torsion free subgroup in $\SU(n,1)$ and $X=\G\backslash\H$ the corresponding
complex hyperbolic manifold, we again call $\pi:\PTX\pfd X$ the projectivized tangent bundle of
$X$. 
The fibration $\pig$, by $\SU(n,1)$-equivariance, defines a smooth holomorphic
foliation by holomorphic curves on $\PTX$.  
This foliation inherits a
structure of transversally homogeneous $\cal G$-foliation, see \cite{God}*{\textsection~III.3
  p.~164}, which will be discussed in \textsection \ref{sec:transverse}. 
  
\medskip

If $\xi\in\PTX$, the leaf $\cal L_\xi$  through $\xi$
is locally given by the holomorphic tangent space of the local complex geodesic tangent to $\xi$ at
$\pi(\xi)$ in $X$. Vectors tangent to the leaves of the foliation form a  line subbundle $L$ of
$T_{\PTX}$. Recall that we can pull-back the 
tangent bundle $T_X\pfd X$ to $\PTX$ to obtain a vector bundle 
$\pi^\star T_X\pfd\PTX$ and that the tautological line 
bundle $\OX$ is the subbundle of $\pi^\star  T_X$ defined by 
$$
\OX=\{(u,\xi)\in T_X\times\PTX\,\mid\, u\in\xi\}.
$$ 
By construction the
differential $\pi_\star$ of $\pi$ at $\xi$ maps the fiber $L_\xi$ of $L$ to the line $\xi\subset
T_{X,\pi(\xi)}$. This means that when considered as a morphism 
from $T_{\PTX}$ to $\pi^\star T_X$, $\pi_\star$ realizes an isomorphism between the line subbundle
$L$ of $T_{\PTX}$ and the tautological line subbundle $\OX$ of $\pi^\star T_X$. 

\smallskip

For these reasons the foliation will be called the {\em tautological foliation} of $\PTX$ and will be
denoted by $\T$.  

\subsection{The transverse structure of the tautological foliation}\label{sec:transverse}\hfill

\smallskip

By construction, the tautological foliation $\T$ has a structure of transversally homogeneous
  $\cal G$-foliation, also called a transverse $(\SU(n,1),\cal G)$-structure. In this section, we
  describe this structure and prove the

\begin{prop}\label{prop:transverse}
The homogeneous space $\cal G$ of complex geodesics of $\H$
  admits a $\SU(n,1)$-invariant non-degenerate but indefinite K\"ahler form $\o_{\cal G}$, which is
  unique up to normalization. This form defines a diffuse invariant transverse measure $\mu_{\cal
    G}$ for the tautological foliation $\T$ on the projectivized tangent bundle 
$\pi:\PTX\pfd X$ of a complex hyperbolic manifold $X$. 
The associated current of integration on $\PTX$ satisfies (when suitably normalized)
$$ 
\into\pi^\star\beta = \frac{1}{n!}\,\int_X\beta\wedge\o^{m-1} 
$$   
for any compactly supported 2-form $\beta$ on $X$. 
\end{prop}

This result will allow us to compute numerical invariants on the complex hyperbolic manifold
  $X$ by going up to the projectivized tangent bundle $\PTX$ and integrating along the leaves of $\T$. For
  instance, if $X$ is compact and $\F$ is a coherent sheaf on
  $X$, then, using the definition of the foliated degree given in \textsection \ref{higgsfol}: 
$$  
\deg \F=\frac{1}{n!}\int_Xc_1(\F)\wedge \o^{n-1}=\into\pi^\star c_1(\F) =\degog (\pi^\star\F)
$$
In the same way, still assuming that $X=\G\backslash\H$ is compact, if $\rho$ is a representation of
$\G$ in a Hermitian Lie group $G$, then the Toledo invariant of $\r$ is given by 
$$
\tau(\rho):=\frac{1}{n!}\int_X f^\star \o_\Y\wedge \o^{n-1}=\into\pi^\star f^\star\o_\Y
$$
where $f$ is any $\r$-equivariant map from $\H$ to the Hermitian symmetric space $\Y$
associated to $G$, and $\o_\Y$ is the K\"ahler form of $\Y$ normalized so that the minimal value of its
holomorphic sectional curvature is $-1$.

\medskip

The existence of the indefinite Kähler form $\o_{\cal G}$ on $\cal G$ stated in
  Proposition~\ref{prop:transverse} is not new, see~\cite{Wolf}*{Theorem 6.3 and Corollary
    6.4}. However we need the correct normalization constants between $\o_{\cal G}$, the invariant
  K\"ahler form $\o$ on $X$, and the curvature of the tautological line bundle on $\PTX$. To work
  these constants out we now describe the geometry of the double holomorphic fibration 
between the complex $\SU(n,1)$-homogeneous 
spaces $\PTH$, $\cal G$ and $\H$ and the (pseudo-)K\"ahler structure of these spaces on the Lie
algebra level. The results are summarized in Lemma~\ref{transversesymplectic} below. The
end of the proof of Proposition~\ref{prop:transverse} is given in Lemma~\ref{compute}.

\medskip

To lighten the notation, in this section we set 
$\stab={\rm S}(\U(n-1)\times\U(1)\times\U(1))$, 
$H={\rm S}(\U(n-1)\times\U(1,1))$, 
and we denote their respective Lie algebras by 
$\mg$ and $\hg$. 

\smallskip

The Lie algebra $\su(n,1)$ of the group $\SU(n,1)$ is:
$$
\su(n,1)=\left\{ 
\left(
\begin{array}{cc}
A & \xi  \\
\xi^\star & a  
\end{array}
\right) ,\, A\in M_n(\C),\,\xi\in\C^n,\, a\in\C\,\mid\, A^\star=-A,\, a+\tr A=0 
\right\}.
$$
The Lie algebra of the maximal compact subgroup of $\SU(n,1)$, isomorphic to  $\U(n)$, is the subalgebra 
$$
\left\{ 
\left(
\begin{array}{cc}
A & 0  \\
0 & a  
\end{array}
\right) ,\, A\in M_n(\C),\, a\in\C\,\mid\, A^\star=-A,\, a+\tr A=0 
\right\}\simeq \u(n),
$$
whereas 
$$\mg=\left\{ 
\left(
\begin{array}{ccc}
A & 0 & 0 \\
0 & a & 0 \\
0 & 0 & b
\end{array}
\right),\, A\in\u(n-1),\, a,b\in\C, a+b+\tr A=0 
\right\}.
$$
$$ \hg=\left\{ 
\left(
\begin{array}{cc}
A & 0  \\
0 & B
\end{array}
\right),\, A\in\u(n-1),\, B\in\u(1,1), \,\tr A+\tr B=0 
\right\}.
$$
The real tangent space of $\PTH$ at $\stab$ is naturally identified with the subspace 
$$
\s=\left\{\xi=
\left(
\begin{array}{ccc}
0 & \xi_3 & \xi_2 \\
-\xi_3^\star & 0 & \xi_1 \\
\xi_2^\star & \xi_1^\star & 0
\end{array}
\right),\, \xi_1\in\C,\, \xi_2,\xi_3\in\C^{n-1} 
\right\}\subset\su(n,1)
$$
and its invariant complex structure $J$ is given at $\stab$ by
$$
J\left(
\begin{array}{ccc}
0 & \xi_3 & \xi_2 \\
-\xi_3^\star & 0 & \xi_1 \\
\xi_2^\star & \xi_1^\star & 0
\end{array}
\right)
=\left(
\begin{array}{ccc}
0 & \rmo\xi_3 & \rmo\xi_2 \\
\rmo\xi_3^\star & 0 & \rmo\xi_1 \\
-\rmo\xi_2^\star & -\rmo\xi_1^\star & 0
\end{array}
\right).
$$ 
Define the subspaces   
$$ \s_1=
\left\{
\left(
\begin{array}{ccc}
0 & 0 & 0 \\
0 & 0 & \xi_1 \\
0 & \xi_1^\star & 0
\end{array}
\right),\, \xi_1\in\C
\right\},
$$
$$
\s_2=\left\{
\left(
\begin{array}{ccc}
0 & 0 & \xi_2 \\
0 & 0 & 0 \\
\xi_2^\star & 0 & 0
\end{array}
\right),\,  \xi_2\in\C^{n-1} 
\right\},
$$
$$
\s_3=\left\{
\left(
\begin{array}{ccc}
0 & \xi_3 & 0 \\
-\xi_3^\star & 0 & 0 \\
0 & 0 & 0
\end{array}
\right),\, \xi_3\in\C^{n-1} 
\right\}
$$
of $\s$. It
is plain that $\u(n)\oplus(\sg_1\oplus\sg_2)$ is a Cartan decomposition of $\su(n,1)$ so that
  $\s_1\oplus\s_2$ is invariant under the adjoint action of $\U(n)$ and identifies with the tangent
  space to $\H=\SU(n,1)/\U(n)$ at $\U(n)$. Similarly, the 
  subspace $\s_2+\s_3$ is $H$-invariant and identifies with the tangent space of $\cal G=\SU(n,1)/H$ at
  $H$. The subspaces 
$\s_1$, $\s_2$ and $\s_3$ are invariant under the adjoint action of $\stab$ on $\s$, and therefore
define $C^\infty$ subbundles of the real tangent bundle of $\PTH$. The subbundle
  corresponding to $\s_1$, resp. $\s_3$, is the tangent bundle of the fibers of $\pig:\PTH\pfd \cal G$,
  resp. $\pi:\PTH\pfd\H$.

\medskip

Let $\o_1$, $\o_2$, $\o_3$ be the skew-symmetric $\R$-bilinear forms on $\sg$ given by
$$
\o_j(\xi,\eta)=2\rmo(\eta_j^\star \xi_j-\xi_j^\star \eta_j),
$$ 
for $\xi=(\xi_1,\xi_2,\xi_3)$ and $\eta=(\eta_1,\eta_2,\eta_3)$ in $\s$. These forms
are invariant by $\stab$ hence they define $\SU(n,1)$-invariant 2-forms on $\PTH=\SU(n,1)/\stab$ which will
be denoted by 
the same letters.

\begin{lemma}\label{transversesymplectic}
\begin{itemize}
\item The bilinear form $\o_1+\o_2$ defines the $\SU(n,1)$-invariant K\"ahler form $\o$ on $\H$
  normalized so as to have constant holomorphic sectional curvature $-1$. 
\item The bilinear form $\o_2-\o_3$ defines a $\SU(n,1)$-invariant non degenerate but indefinite
  K\"ahler form $\o_{\cal G}$ on the space of complex geodesics $\cal G$ of $\H$. 
\item The bilinear form $\o_1+\frac 12\,(\o_2+\o_3)$ defines a K\"ahler form $\pi^\star\o-\frac
  12\,\pig^\star\o_{\cal G}$ on the projectivized
  tangent bundle $\PTH$ of $\H$. It is the curvature form of the dual ${\mathcal O}_{\PTH}(1)$ of
  the tautological line bundle over $\PTH$ endowed with the natural metric induced by $T_\H$.
\end{itemize}
\end{lemma}

\begin{proof} 
It is easily checked that the bilinear form
$\o_1+\o_2$ on $\sg_1\oplus\sg_2$ is invariant by $\U(n)$, hence that it defines a $\SU(n,1)$-invariant 2-form
$\o$ on $\H=\SU(n,1)/\U(n)$. The  
form $\o$ is closed (for example because it is a 2-form on a symmetric space and it is invariant by
the geodesic symmetries) and it is precisely the
invariant K\"ahler form on $\H$, normalized so as to have constant holomorphic sectional curvature
$-1$. It is also given by $\o(\xi,\eta)=b(\zeta,[\xi,\eta])=b({\rm ad}(\zeta)\xi,\eta)$ for
$\xi,\eta\in\sg_1\oplus\sg_2$, where $b$ is the Killing form on $\su(n,1)$
and $\zeta$ is the element of the 1-dimensional center of $\u(n)$ such that $\ad(\zeta)$ gives the invariant
complex structure of $\H$:       
$$
\zeta=\left(
\begin{array}{cc}
\frac{\rmo}{n+1}1_{n} & 0 \\
0 & -\frac{n\rmo}{n+1}  
\end{array}
\right)
$$
(Here and in the rest of the paper, if $k$ is an integer, $1_k$ denotes the identity matrix of size $k$.) 

\medskip

One also checks that the bilinear form $\o_2-\o_3$ on $\sg_2\oplus\sg_3$ is
invariant by $H$ and hence 
defines a $\SU(n,1)$-invariant form $\o_{\cal G}$ on $\cal G=\SU(n,1)/H$ that is indeed non degenerate (its
signature is $(n-1,n-1)$). Again,
this form can be computed as  $\o_{\cal 
  G}(\xi,\eta)=b(\zeta_\hg,[\xi,\eta])=b({\rm ad}(\zeta_\hg)\xi,\eta)$ for 
$\xi,\eta\in\sg_2\oplus\sg_3$, where  
$$
\zeta_\hg=\left(
\begin{array}{cc}
\frac{2\rmo}{n+1}1_{n-1} & 0 \\
0 & -\frac{(n-1)\,\rmo}{n+1}1_2  
\end{array}
\right)
$$
is the element of the 1-dimensional center of $\hg$ such that $\ad(\zeta_\hg)$ gives the invariant
complex structure of $\cal G$. The $\SU(n,1)$-invariance of the Killing form $b$ and the
  Jacobi identity imply that $\o_{\cal G}$ is
closed. 

\medskip

Endow the dual ${\mathcal O}_{\PTH}(1)$ of the tautological
line bundle over $\PTH$  with 
the natural metric induced from the one of $T_\H$. Its curvature form
is a positive
(1,1)-form and one can compute (see e.g.~\cite{GriffithsSchmid}*{Part 4} or~\cite{CW}*{(3.7)}) that: 
$$
\sqrt{-1}\Theta({\mathcal
  O}_{\PTH}(1))=\o_1+\frac{1}{2}(\o_2+\o_3)=(\o_1+\o_2)-\frac 12\,(\o_2-\o_3)=\pi^\star\o-\frac
12\,\pig^\star\o_{\cal G}.
$$
(Note  that we normalized the metric on $T_\H$ in order to have constant holomorphic
sectional curvature $-1$ and that $\om_3$ restricted to a fiber of $\pi$ is $2\,\om_{FS}$ in~\cite{CW}.) 
\end{proof}

\medskip

If $X=\G\backslash\H$ is a complex hyperbolic manifold, the $\SU(n,1)$-invariant
(1,1)-forms $\o$ on $\H$ and $\pig^\star\o_{\cal G}$ on $\PTH$ defined in Lemma~\ref{transversesymplectic} descend to closed forms on $X$ and $\PTX$
respectively which will be denoted by the same letters. The form $\o$ is the K\"ahler form of
$X$. The K\"ahler form $\pi^\star\o-\frac 12\,\pig^\star\o_{\cal G}$ is the curvature form of
${\mathcal O}_{\PTX}(1)$, which is isomorphic to the dual $L^\vee$ of the tangent bundle $L$ to the
tautological foliation $\T$ on $\PTX$.

The non degenerate indefinite K\"ahler form $\o_{\cal G}$ on $\cal G$ defines an invariant transverse K\"ahler
 form for the foliation $\T$ on the projectivized tangent bundle
$\PTX$ of $X$, hence an invariant transverse volume form and an invariant transverse measure $\mu_{\cal G}$ for the
foliation. We normalize the induced current of 
integration along the leaves of $\T$ so that for all compactly supported 2-form $\alpha$ on $\PTX$, 
$$
\into\alpha=\int_{\PTX}\alpha\wedge\Omegag,
$$     
where
$$
\Omegag:=\frac{(-1)^{n-1}}{(2n-2)!\,\vol(\CP^{n-1})}\,\pig^\star \o_{\cal G}^{2n-2}.
$$ 
The form $\Omegag$ is a closed basic semi-positive $(2n-2,2n-2)$-form of rank $4n-4$ on $\PTX$.

It is now easy to complete the proof of Proposition~\ref{prop:transverse}:

\begin{lemma}\label{compute} 
Let $\b$ be a compactly supported 2-form on $X$, then 
$$ \frac{1}{n!}\int_X \b\wedge\o^{n-1} =\int_\PTX \pi^\star\b\wedge\Omegag.
$$
\end{lemma}

\begin{proof}
By $\SU(n,1)$-invariance the (1,1)-forms $\o_1$, $\o_2$ and $\o_3$ on $\PTH$ descend to forms on
$\PTX$ and we again have $\pig^\star\o_{\cal G}=\o_2-\o_3$. Let $\alpha$ be a compactly supported
$2$-form on $\PTX$. Then, 
$$
\begin{array}{rcl}
\ds \int_\PTX\a\wedge\o_2^{n-1}\wedge\o_3^{n-1} 
& = & \ds\frac{1}{2}\int_\PTX\la
\alpha,\o_1\ra\,\o_1\wedge\o_2^{n-1}\wedge\o_3^{n-1} \\
& = & \ds\frac{1}{2n}\int_\PTX\la \alpha,\o_1\ra\,(\pi^\star\o)^{n}\wedge\o_3^{n-1}\\
& = & \ds\frac{1}{2n}\int_X\left(\int_{\pi^{-1}(x)}\la \alpha,\o_1\ra\,\o_3^{n-1} \right)\o^n.
\end{array}
$$
If now $\beta$ is a compactly supported 2-form on $X$, one has
$$\int_{\pi^{-1}(x)}\la\pi^\star\beta,\o_1\ra\,\frac{\o_3^{n-1}}{(n-1)!}=\frac{{\rm
    vol}(\CP^{n-1})}{n}\la \beta,\o\ra_x.
$$
Hence
$$
\begin{array}{rcl}
\ds \int_\PTX\pi^\star\beta\wedge\o_2^{n-1}\wedge\o_3^{n-1} 
& = & \ds \frac{(n-1)!\,{\rm vol}(\CP^{n-1})}{2n^2}\int_X\la\beta,\o\ra\,\o^n \\
& = & \ds (n-1)!^2\,{\rm vol}(\CP^{n-1})\,\frac{1}{n!}\int_X\beta\wedge\o^{n-1},
\end{array}
$$
so that
$$
\begin{array}{rcl}
\ds  \frac{1}{n!}\int_X \b\wedge\o^{n-1} & = & \ds \frac{1}{{\rm vol}(\CP^{n-1})}\int_\PTX\pi^\star
\b\wedge\frac{\o_2^{n-1}}{(n-1)!}\wedge\frac{\o_3^{n-1}}{(n-1)!} \\
& = & \ds \frac{(-1)^{n-1}}{{\rm vol}(\CP^{n-1})}\int_\PTX\pi^\star
\b\wedge \frac{(\o_2-\o_3)^{2n-2}}{(2n-2)!}.
\end{array}
$$
\end{proof}

\begin{rema}\label{closedleaves2}
{\em Closed leaves II: convergence of currents.}  As we said in Remark~\ref{closedleaves1},
  the current of integration given by a closed leaf of a foliation is not in general
  well-behaved. 
 This is still true in the case of the tautological foliation on 
   the projectivized tangent bundle $\PTX$ of a complex hyperbolic manifold $X$. 
More importantly here, there is no direct relation, such as the one established in
Proposition~\ref{prop:transverse}, between the integration along a single closed leaf of the
tautological foliation and integration against $\o^{n-1}$.  
It is however possible  to exploit the existence of closed totally geodesic curves $C_i$ in
$X=\G\backslash\H$ to make such a relation, when there 
are infinitely many such curves. 
This is true for example if $\G$ is a so-called arithmetic lattice of type I of $\SU(n,1)$,
  and moreover in this case the 
  sequence of curves $C_i$ can be chosen so that no subsequence is contained in a proper totally
  geodesic submanifold of $X$. As is proved e.g. in~\cite{KozMobequidistrib}, this implies that the
currents $\int_{C_i}$ suitably normalized converge towards $\omega^{n-1}$. 
\end{rema}

\subsection{Some consequences of Ratner's theorem on orbit closures}\label{sec:closure}\hfill

\smallskip

We just saw that the tautological foliation $\T$ on the projectivized tangent bundle $\PTX$ of
  a complex hyperbolic manifold $X$ has a rich {\em transversal}
  structure. We consider now the {\em tangential} structure of the foliation $\T$ and  
  we state fundamental properties of its leaves which follow from the resolution by M. Ratner of
  Raghunathan's conjecture on orbit closures. 

\smallskip

In this section we come back to the setting of the paper so that $\G$ is a torsion free
  uniform lattice of $\SU(n,1)$ and $X=\G\backslash\H$ is therefore a compact complex hyperbolic
  manifold. 

\medskip

Let $\cal L$ be a leaf of the tautological foliation $\T$ on $\PTX=\G\backslash
\SU(n,1)/\stab$. In this section again, $\stab$ is short for 
${\rm S}(\U(n-1)\times\U(1)\times\U(1))$. The leaf $\cal L$ is of the form 
$\G\backslash \G \,U_\L\, g_\L \stab/\stab$ for some $g_\L\in \SU(n,1)$ and a group $U_\L$ locally
isomorphic to 
$\SU(1,1)$. Because $U_\L$ is generated by unipotent elements, it follows from the work of
Ratner~\cite{RatnerTopo} 
that the closure of the orbit $\G e\cdot U_\L$ in $\G\backslash \SU(n,1)$ is homogeneous, namely
that there exists a closed subgroup $S_\L$ of $\SU(n,1)$ such that $U_\L\subset S_\L$ and 
  $\overline{\G e\cdot U_\L}=\G e\cdot S_\L$. This implies that $\G\cap S_\L$ is a
  lattice in $S_\L$~\cite{Raghunathan}*{Theorem~1.13} and that $S_\L$ is a reductive group with
  compact center,  for example because $\rk_\R \SU(n,1)=1$~\cite{Shah}.   

By~\cite{P}, the fact that $\rk_\R U_\L=\rk_\R \SU(n,1)$ implies that the Lie algebra of $S_\L$ is stable
by the Cartan involution of $\SU(n,1)$ given by the point $g_\L \U(n)$ of $\H=\SU(n,1)/\U(n)$, so
that the orbit $\tilde Y_\L:=S_\L\cdot g_\L \U(n)$ of $g_\L \U(n)$ under $S_\L$ in $\H$ is a totally
geodesic submanifold and $Y_\L:=\G\backslash\G\tilde Y_\L$ is a closed 
immersed totally geodesic submanifold of $X=\G\backslash \SU(n,1)/\U(n)$. This submanifold is the closure of
the projection $\pi(\L)$ of $\L$ in $X$.

Summing up, we have:

\begin{prop}\label{prop:closure}
  Let $X$ be a compact complex hyperbolic manifold and let $\L$ be a leaf of
  the tautological foliation $\T$ on $\PTX$. The closure $\overline{\pi(\L)}$ 
  of the image of $\L$ 
  by the projection $\pi:\PTX\pfd X$ is a closed immersed totally geodesic submanifold of $X$.   
\end{prop}

The last proposition has
the following consequence on projections to $X$ of $\T$-saturated subsets of $\PTX$:
 
\begin{prop}\label{prop:dyn}
  Let $X$ be a compact complex hyperbolic manifold and let $S$ be a closed
  $\T$-saturated proper subset of $\PTX$. Then $\pi(S)$ is a proper subset of $X$. 
\end{prop}

\begin{proof}
The key point is that there is at most a countable number of closed immersed totally geodesic
submanifolds of dimension at least 2 in $X$. This follows
from~\cite{RatnerMeasure}*{Theorem~1.1} but in our case a similar but simpler argument is available.

Let $Y\subset X$ be such a submanifold. This
  means that $Y=\G\backslash\G\tilde Y$ where $\tilde Y$ is a symmetric subspace of the noncompact
  type of $\H$ whose stabilizer $S$ 
  in $\SU(n,1)$ contains $\Lambda:=\G\cap S$ as a lattice. Moreover, there exists $y\in \H$ and a
  simple (because $\rk_\R \SU(n,1)=1$) noncompact subgroup $H$ of $S$ such that $\tilde
Y=S\cdot y= H\cdot y$.

We claim that $\tilde Y$, and hence $Y$, is entirely determined by the intersection $\Lambda=\G\cap
S$. Indeed, let $Y'$ be another closed immersed totally geodesic submanifold of dimension at
least 2 of $X$, let $\tilde Y'$, $S'$, $\Lambda'$, $H'$ and $y'$ be defined as above for $Y$, and
assume that $\Lambda'=\Lambda$.   

By a strengthening
of the Borel density theorem, see e.g.~\cite{Dani}*{Corollary~4.2}, since $\Lambda$
is a lattice in $S$ and $H$ is a  simple noncompact subgroup of $S$, the Zariski closure 
$\ov{\Lambda}^{\,\textsf z}$  of $\Lambda$ in $\SU(n,1)$ contains $H$. Therefore
$H\subset\ov{\Lambda}^{\,\textsf z}=\ov{\Lambda'}^{\,\textsf
  z}\subset S'$ because $S'$ is Zariski-closed. In the same way,
$H'\subset S$.  

If $d$ denotes the
distance function on $\H$, the function
 $x\mapsto d(x,\tilde Y')$ is constant on $\tilde Y$ because
$\tilde Y$ is an $H$-orbit, $\tilde Y'$ is an $S'$-orbit and $H\subset S'$. It must be identically zero, because if not, the convex hull of two distinct points in $\tilde Y$ and
their (distinct) projections in $\tilde Y'$ is Euclidean by the flat quadrilateral
theorem~\cite{BH}*{p.~181}, a contradiction since $\rk\, \H=1$.  Hence $\tilde
Y\subset\tilde Y'$. The same reasoning gives $\tilde Y'\subset\tilde Y$. This
is what we wanted. 

Since $\Lambda=\G\cap S$ is finitely generated because it is a lattice in $S$, and since there are
only countably many finite subsets in $\G$, this indeed proves that there are at most countably many closed
immersed totally geodesic submanifolds of dimension at least 2 in $X$.  

\smallskip

To conclude,  let $S$ be a closed $\T$-saturated proper subset of $\PTX$ and assume that $\pi(S)=X$. Then
because $S$ 
is a union of leaves, $X$ is the union of the projections of the leaves of $S$, hence of their
closures. By Proposition~\ref{prop:closure}, these closures are closed immersed totally geodesic
submanifolds of $X$ of dimension at least 2. Since there are only countably many such objects,  
there must be a leaf $\L\subset S$ such that $\overline{\pi(\L)}=X$. But
  $\overline{\pi(\L)}$ is the projection to $X$ of the totally geodesic orbit $\tilde Y_\L=S_\L\cdot
g_\L\U(n)$ in $\H$, so this orbit must be the whole $\H$, and $S_\L$ being reductive,
this implies that $S_\L=\SU(n,1)$,
so that $\overline \L=\PTX$. Hence $S=\PTX$, for $S$ is closed. A contradiction.
\end{proof}

\section{Representations in $\SU(p,q)$, $p\geq q$}\label{max}

\subsection{Strategy of the proof}\label{strategy}\hfill 

\smallskip

This section is devoted to the proof of Theorem~\ref{main} and Corollary~\ref{description} in
the case of representations in the group $\SU(p,q)$. 
Representations in the other classical Hermitian Lie groups will be treated in \textsection
\ref{sec:others} using results of this section. 

Our primary goal here is to prove:

\begin{theo}\label{thm:maxreduc}
  Let $\r$ be a reductive representation of a (torsion free) uniform lattice
  $\G$ of $\SU(n,1)$ in $\SU(p,q)$, $p\geq q\geq 1$. If $\r$ is maximal and $n\geq 2$ then the
  $\r$-equivariant harmonic map $f$ from 
  $\H$ to the symmetric space $\Y_{p,q}$ of $\SU(p,q)$ is holomorphic or antiholomorphic.
\end{theo}

We will explain in \textsection \ref{sec:end} why Theorem~\ref{main} and
Corollary~\ref{description} for reductive representations in $\SU(p,q)$ follow from this result. 

As we said in the 
introduction, it is a theorem of~\cite{BIW09} that maximal representations are necessarily
reductive, so that non reductive ones could be excluded from the very beginning. We will
nevertheless discuss (and indeed exclude, eventually)  
non reductive representations in \textsection \ref{nonreductive}, for two reasons. Firstly, the
arguments of~\cite{BIW09} are quite different from those of the present paper, and we wish to be as 
self-contained as possible. Secondly and more importantly, it is interesting to see how
the rigidity of reductive maximal representations in turn implies that non reductive
ones don't exist.   

\medskip

Our approach to Theorem~\ref{thm:maxreduc} is based on the study of the real Higgs bundle $(E,\t)$ over the
compact complex hyperbolic 
manifold $X=\Gamma\backslash\H$ constructed from the $\r$-equivariant harmonic map
$f:\H\pfd\Y_{p,q}$ (which exists since $\r$ 
is reductive), see \textsection \ref{higgs}. 

After some preliminaries on the group $\SU(p,q)$ and its symmetric space $\Y_{p,q}$, the {\em real
structure}  of the Higgs bundle $(E,\t)$ will be described in \textsection \ref{sec:real}. We shall
see that $E$ is a  
direct sum $V\oplus W$ and that the Higgs field $\t$ has two components $\b:W\otimes T_X\pfd V$ and
$\g:V\otimes T_X\pfd W$ corresponding respectively to the holomorphic and antiholomorphic parts of
the $\r$-equivariant harmonic map $f$, so that $f$ is holomorphic, resp. antiholomorphic, if and
only if $\g=0$, resp. $\b=0$.    

\medskip

The proof of Theorem~\ref{thm:maxreduc} then proceeds
in three steps. 

\smallskip

The first step is a new proof of the Milnor-Wood inequality 
obtained by pulling back the Higgs
bundle $(V\oplus W,\b\oplus\g)$ over $X$ to a Higgs bundle
$(\tilde V\oplus \tilde W,\tilde\b\oplus\tilde\g)$ over the projectivized tangent bundle $\PTX$ with
its tautological foliation 
$\T$, so that the ideas concerning foliated Higgs bundles developed in \textsection
\ref{higgsfol} and  
\textsection \ref{sec:transverse} can come into play.  

Proposition~\ref{prop:transverse} is used to show that the Milnor-Wood inequality is equivalent to
an inequality between the foliated degrees of certain bundles on $\PTX$, namely
$$
|\degog \tilde W|\leq q\,\,\frac{\degog L^\vee}{2}
$$ 
where $L^\vee$ is the dual of the tangent line bundle $L$ to the foliation $\T$. 

Thanks to the semistability statement of Proposition~\ref{dec}, and because the leaves of $\T$ are
complex curves, a more precise statement (Proposition~\ref{prop:MW}) is proved, exactly in
the same way as   
for surface groups, i.e. lattices in $\SU(1,1)$, see e.g.~\cites{Xia,MarkmanXia,BGPG1}. 

\smallskip

The second step is an analysis of the singular loci of the components $\b$
and $\g$  of the
Higgs field. Consider $\b$ for example. Define the singular locus $\S_{\tilde\b}$ of
  $\tilde\b$ as the following subset of $\PTX$:    
$$
\S_{\tilde\b}=\{\xi\in\PTX\,\mid\,\rk\,\tilde\b_\xi<\rk\,\tilde\b\}
$$
where $\rk\,\tilde\b$ is the generic rank of $\tilde\b:\tilde W\otimes L\pfd \tilde V$, and the
singular locus 
$\S_\b$ of $\b$ as the projection $\pi(\S_{\tilde\b})$ of $\S_{\tilde\b}$ to $X$.   

The set $\S_{\tilde\b}$ is a proper analytic subset of $\PTX$. We want to prove that
when the 
inequality of Proposition~\ref{prop:MW} is an equality, and say 
$\degog \tilde W>0$, then  $\S_\b$ is also a proper analytic subset in $X$.  
This is achieved by first proving  that $\S_{\tilde\b}$ is saturated under the tautological
foliation $\T$. This follows  
from our proof of the inequality, and from the weak polystability of $(\tilde E,\tilde \t)$ along
the leaves of $\T$, see Proposition~\ref{dec}.  
We may then apply Proposition~\ref{prop:dyn} which indeed implies that $\S_\b$ is a proper
analytic subset of $X$. 
    
\smallskip

The third step is the conclusion where we prove that in the maximal case, and say if $\tau(\r)>0$,
the injectivity of $\b_x(\xi)$ for all 
$x \in X\backslash \S_\b$ and $\xi\in T_{X,x}\backslash\{0\}$ forces $\g$ to vanish, hence the
$\rho$-equivariant harmonic map to be holomorphic. Here we use the integrability condition $[\t,\t]=0$ and our standing assumption that $n\geq 2$.      
 If $\tau(\r)<0$, the harmonic map is
  proved to be antiholomorphic by considering $\g$ instead of $\b$.

\subsection{Preliminaries} \label{sec:prelim} \hfill

\subsubsection{The symmetric space $\Y_{p,q}$}\label{sec:Ypq}\hfill

\smallskip

We recall here some necessary facts on the symmetric space $\Y_{p,q}$ associated to the group
$\SU(p,q)$.
We refer to~\cite{KozMobRank2}*{\textsection 3.1} for details.   

\medskip
 
Let $\E$ be the vector space $\C^{p+q}$ equipped with a Hermitian form $h_{p,q}$ of signature
$(p,q)$, $p\geq q$. The group $\SU(p,q)=\SU(\E,h_{p,q})$ acts transitively on $\Y_{p,q}$, the open subset of
the Grassmann manifold of $q$-dimensional subspaces of $\E$ consisting of $q$-subspaces on which
$h_{p,q}$ restricts to a negative definite Hermitian form. Let $\W\subset\E$ be a point in
$\Y_{p,q}$, and $\V\subset\E$ be its orthogonal complement w.r.t. $h_{p,q}$. 
The stabilizer of $\W$ is a maximal compact subgroup of $\SU(p,q)$ and is isomorphic to ${\rm
  S}(\U(p)\times\U(q))$. Hence $\Y_{p,q}=\SU(p,q)/{\rm S}(\U(p)\times\U(q))$. As a bounded symmetric
domain, it is naturally identified with $\{Z\in  
M_{p,q}(\C)\,|\,1_q-Z^\star Z>0\}\subset\C^{pq}$. The rank of the symmetric space 
  $\Y_{p,q}$ is $\min\{p,q\}=q$.

We have the tangent spaces identifications 
$T_{\Y_{p,q},\W}\simeq T^{1,0}_\W \Y_{p,q}\simeq \W^\star \otimes\V$ and $T^{0,1}_\W \Y_{p,q}\simeq \V^\star\otimes \W$. 
We normalize the $\SU(p,q)$-invariant metric $\o_{\Y_{p,q}}$ on $\Y_{p,q}$ so that, representing an element of
$T_{\Y_{p,q},\W}$ by a matrix in  
$M_{p,q}(\C)$, the holomorphic sectional
curvature for the complex line $\la A\ra$ generated by a nonzero $A\in M_{p,q}(\C)$ is given by 
$$
\kappa(\la A\ra)=-\frac{\tr\bigl(\left(A^\star A\right)^2\bigr)}{\left(\tr\left(A^\star A\right)\right)^2}.
$$
This formula shows that $\kappa(\la A\ra)$ is pinched between $-1$ and $-1/q$ and that $\kappa(\la
A\ra)=-1/q$ 
if and only if the column vectors of $A$ are pairwise orthogonal and have the same norm (for the
standard Hermitian scalar product in $\C^p$). 

The symmetric space $\Y_{p,q}$ is a K\"ahler-Einstein manifold, and with our curvature normalization,
the first Chern form of  its holomorphic tangent bundle $T_{\Y_{p,q}}$ is
$c_1(T_{\Y_{p,q}})=-\frac{1}{2\pi}\,\frac{p+q}{2}\,\o_{\Y_{p,q}}$.

\subsubsection{The real structure of a $\SU(p,q)$-Higgs bundle}\label{sec:real}\hfill

\smallskip

Harmonic Higgs bundles arising from reductive representations into real
reductive subgroups $G$ of 
$\SL(N,\C)$ have an additional {\em real structure},  compared to
those arising from representations in $\SL(N,\C)$ without further restriction,
see  e.g.~\cite{S2}*{p.~90-91} or~\cite{Maubon}*{\textsection 3.6}.

We describe this real structure in our case, namely for $G=\SU(p,q)\subset\SL(p+q,\C)$. More details
can be found for example in~\cite{KozMobRank2 }*{\textsection 2.4 \& \textsection 3.2} or
\cite{Maubon}*{\textsection 3.6.2 \& \textsection 3.6.3}. 

The first observation is that a Higgs bundle $(E,\t)\pfd \mf$ associated to a reductive
representation in $\SU(p,q)$ of the 
fundamental group of a compact K\"ahler manifold $\mf$ splits 
holomorphically as a sum $E=V\oplus W$, where $V$ has rank $p$ and $W$ has rank $q$. Indeed, as
  a smooth bundle, $E$ is the flat bundle associated to the standard representation of $\SL(p+q,\C)$ on
$\E=\C^{p+q}$. But the
$\r$-equivariant harmonic map $f:\tilde\mf\pfd\Y_{p,q}$ defines a reduction of its structure group to the
maximal compact subgroup ${\rm S}(\U(p)\times\U(q))$ of $\SU(p,q)$, hence to its complexification ${\rm
  S}(\GL(p,\C)\times\GL(q,\C))$. Since this group preserves the decomposition $\E=\V\oplus\W$, we
indeed get the holomorphic splitting $E=V\oplus W$. 

Note that $\deg V+\deg W=\deg E=0$  because $E$ is flat, and that since the harmonic metric on $E$ is
defined by the reduction of the structure group of $E$ to ${\rm S}(\U(p)\times\U(q))$, the
direct sum $E=V\oplus W$ is orthogonal for the harmonic metric.

The Higgs field is by construction the $(1,0)$-part $d^{1,0}f:T^{1,0}\tilde\mf\pfd T^\C\Y_{p,q}$
of the complexified differential of the harmonic map $f$, seen as an endomorphism of $E$ in the
following way.  The
complexified tangent bundle $T^\C\Y_{p,q}$ is the bundle on $\Y_{p,q}=\SU(p,q)/{\rm
  S}(\U(p)\times\U(q))$ associated to the 
adjoint action of ${\rm S}(\GL(p,\C)\times\GL(q,\C))$ on
$T^\C_\W\Y_{p,q}=T^{1.0}_\W\Y_{p,q} \oplus T^{0.1}_\W\Y_{p,q}
=(\W^\star\otimes\V)\oplus(\V^\star\otimes\W)$. Therefore the  pull-back bundle $f^\star T^\C\Y_{p,q}$
over $\mf$ is the subbundle  $(W^\star\otimes V)\oplus(V^\star\otimes W)$ of $\End(E)$ and the Higgs
field $\t$ is $d^{1,0}f$
seen as a holomorphic $1$-form with values in 
this bundle.

This means that  the Higgs field seen as a sheaf morphism 
$\t:E\otimes T_\mf\pfd E$ is off-diagonal w.r.t. the decomposition $E=V\oplus W$: it has two components
$\b:W\otimes T_\mf\pfd V$ and $\g:V\otimes T_\mf\pfd W$.  Moreover the vanishing of $\g$,
resp. $\b$, exactly means that $d^{1,0}f:T^{1,0}\tilde\mf\pfd T^\C\Y_{p,q}$ 
maps $T^{1,0}\tilde Y$ to $T^{1,0}\Y_{p,q}$, resp. $ T^{0,1}\Y_{p,q}$, i.e. that the harmonic map $f:\tilde\mf\pfd\Y_{p,q}$ is
holomorphic, resp. antiholomorphic. 

\subsection{Proof of Theorem~\ref{thm:maxreduc}}\label{sec:reduc}\hfill

\smallskip
 
In this section $\G$ is a torsion free uniform
complex hyperbolic lattice $\G$ in $\SU(n,1)$, $X=\G\backslash\H$ the corresponding compact complex
hyperbolic manifold, $\pi:\PTX\pfd X$ the projectivized tangent bundle of $X$,  $\T$ the tautological
foliation on $\PTX$, $\r$ a reductive representation of $\G$ in $\SU(p,q)$ and  
$(E,\t)=\ds(V\oplus W,\left(\begin{array}{cc} 0 & \b\\ \g & 0\end{array}\right))$ the
$\SU(p,q)$-Higgs bunde over $X$ associated to $\r$. 

\medskip

We consider the pull-back $(\pi^\star E,\pi^\star\t)$ of the Higgs bundle $(E,\t)\pfd X$ to the
projectivized tangent bundle $\PTX$. This is the harmonic Higgs bundle over $\PTX$ associated to the representation of $\pi_1(\PTX)\simeq\pi_1(X)$ induced by $\rho$.  
To lighten the notation, pulled-back objects will be denoted with a ``$\sim$''. In
particular, $\tilde W$ is the rank $q$ holomorphic bundle $\pi^\star W$ on $\PTX$, $\tilde V$ is the rank $p$
holomorphic bundle $\pi^\star V$, and 
$$
\left\{
  \begin{array}{l}
 {\tilde\beta}:\tilde W\otimes L\fd \tilde V\\
{\tilde\gamma}:\tilde V\otimes L\fd \tilde W   
  \end{array}
\right.
 $$ 
are the two components of the lifted Higgs field $\tilde \t$ restricted to the tangent line bundle $L$ of
the tautological foliation ${\T}$ on $\PTX$.  
(From now on, we shall denote by
the same letter a vector bundle defined on $\PTX$, or $X$, and the sheaf of its sections.) 

\medskip

Summing up, on the projectivized tangent bundle $\PTX$, we have the harmonic $\SU(p,q)$-Higgs bundle
$(\tilde E,\tilde \t)$ and the tautological
foliation $\T$ with its invariant transverse measure $\mu_{\cal G}$ given by the transverse indefinite K\"ahler
form $\o_{\cal G}$, see Proposition~\ref{prop:transverse}. By Proposition~\ref{dec}, $(\tilde
E,\tilde \t)$ is weakly polystable along the leaves of $\T$ with respect to $\mu_{\cal G}$ and we will
exploit this fact to prove 
Theorem~\ref{thm:maxreduc}. 

\medskip

From now on, foliated degrees of sheaves on $\PTX$ will always be computed with the
  transverse measure $\mu_{\cal G}$. Hence we will abbreviate the
  notation $\degog$ to $\dego$.

\subsubsection{Milnor-Wood inequality}\label{sec:ineq}\hfill 

\smallskip

We begin by reformulating
  the Milnor-Wood inequality on the Toledo invariant of $\r$ in 
terms of foliated degrees of vector bundles on $\PTX$.

\medskip

If $f:\H\pfd\Y_{p,q}$ is the $\r$-equivariant harmonic map, the Toledo invariant of $\r$ is given by 
$$
\tau(\r)=\frac{1}{n!}\int_X f^\star\o_{\Y_{p,q}}\wedge \o^{n-1}. 
$$
We saw in \textsection \ref{sec:Ypq} and \textsection \ref{sec:real} that $c_1(
T_{\Y_{p,q}})=-\frac{1}{2\pi}\,\frac{p+q}2\,\o_{\Y_{p,q}}$ and $f^\star T_{\Y_{p,q}}\simeq
W^\star\otimes V$, where $f^\star T_{\Y_{p,q}}$ is here seen as a bundle on $X=\G\backslash\H$ by
$\r$-equivariance of $f$. Remembering that $\deg V=-\deg W$, we get
$$
\tau(\r)=-\frac{4\pi}{p+q}\,\deg f^\star T_{\Y_{p,q}}=-\frac{4\pi}{p+q}\, (-p\,\deg W+q\,\deg V)=4\pi\,\deg W=4\pi\, \dego\tilde W, 
$$
where the last equality is given by Proposition~\ref{prop:transverse}.

On the other hand, by the last item of Lemma~\ref{transversesymplectic} we have 
$c_1(L^\vee)=\frac{1}{2\pi}(\pi^\star\o-\frac 12\,\pi_{\cal G}^\star\o_{\cal G})$, so that again by
Proposition~\ref{prop:transverse} and the definition of the invariant transverse measure $\mu_{\cal G}$ just
before Lemma~\ref{compute}: 
$$
\dego L^\vee
=\frac{1}{2\pi}\int_\PTX (\pi^\star\o-\frac 12\,\pi_{\cal G}^\star\o_{\cal G})\wedge\Omegag
=\frac{1}{2\pi}\int_\PTX \pi^\star\o\wedge\Omegag
=\frac{1}{2\pi \,n!}\int_X \o^n
=\frac 1{2\pi}\,{\vol}(X).
$$
Therefore the Milnor-Wood inequality $|\tau(\r)|\leq q\,\vol(X)$ for reductive representations is equivalent to 
$$
|\dego \tilde W|\leq q\,\,\frac{\dego L^\vee}2
$$
and reductive maximal representations are those for which this inequality is an equality. 
The proof of this inequality will mimic the ``Higgs bundles proof'' of the Milnor-Wood
  inequality in the one dimensional case, i.e. for representations of surface groups (see
  e.g.~\cites{Xia}). It is based on 
  the semistability of the Higgs bundle $(\tilde E,\t)$ along the leaves of $\T$.
  
To be more precise, let $\rk\, \tilde\b$ and  $\rk\,\tilde\g$ be the generic ranks of $\tilde \b:
    \tilde W\otimes L\pfd\tilde V$ and $\tilde \g:\tilde V\otimes L\pfd\tilde W$.
Since $\rk\, \tilde\b$ and  $\rk\,\tilde\g$ are bounded above by $q$, the Milnor-Wood
inequality follows from 

\begin{prop}\label{prop:MW} We have
$\ds -\rk\, \tilde\g\,\,\frac{\dego L^\vee}2\leq \dego \tilde W\leq \rk\, \tilde\b\,\,\frac{\dego L^\vee}2.$
\end{prop}

\begin{proof} 
We first prove that $\dego \tilde W\leq \rk\, \tilde\b\,\,\frac{\dego L^\vee}2$. Consider $\tilde\b:
\tilde W\otimes L\pfd\tilde V$. 

If $\tilde \b=0$ then $\tilde W$ is a leafwise Higgs subsheaf of $(\tilde E,\tilde\t)$ hence by semistability
along the leaves of $\T$, see Proposition~\ref{dec},
$\dego\tilde W\leq 0$ and we are done. 

Assume therefore that $\tilde\b\neq 0$. Let $\cal N=\Ker\tilde\beta\subset \tilde W\otimes L$ and 
$\cal I$ be the saturation (as a sheaf) of $\Im\tilde\beta\subset\tilde V$. By construction, 
$\cal N\otimes L^\vee$ and $\tilde W\oplus \cal I$ are  leafwise Higgs subsheaves of $(\tilde
E,\tilde \t)$ and, again by leafwise semistability, 
$$\dego (\cal N\otimes L^\vee)\leq 0\ \ {\rm and}\ \ \dego\tilde W+\dego \cal I\leq 0.
$$
Moreover, we have $\dego(W\otimes L)=\dego\cal N+\dego \Im\tilde\b\leq\dego\cal N+\dego\cal I$
since $\dego\cal I\geq\dego\Im\tilde\b$, see the proof of Proposition~\ref{dec}~(\ref{semistab}). Because the generic rank of
${\tilde\beta}$ is the rank of $\cal I$ and of $(\tilde W\otimes L)/\cal N$, we thus have 
$$\dego\tilde W + q\,\dego L\leq \dego \cal N+\dego \cal I\leq (q-\rk\,\tilde\b)\,\dego L -\dego\tilde W
$$
and hence $\ds\dego\tilde W\leq  \rk\, \tilde\b\,\,\frac{\dego L^\vee}2$.

In the same way, using $\tilde \g:\tilde V\otimes L\pfd\tilde W$, we obtain
$\ds -\dego\tilde W=\dego\tilde V\leq \rk\, \tilde\g\,\,\frac{\dego L^\vee}2$.
\end{proof}

\begin{rema} 
This proposition holds more generally in the setting of \textsection \ref{higgsfol}, i.e. for a
  harmonic $\SU(p,q)$-Higgs bundle over a compact K\"ahler manifold $Y$ with a smooth foliation $\T$ by
  holomorphic curves and 
an invariant transverse volume form, so that Proposition~\ref{dec}
  applies. Note that in this case $\tilde\beta$ and $\tilde\gamma$ have to be understood
    as the restriction of the components $\beta$ and $\gamma$ of the Higgs field to the leaves. 
  
  We remark that by the proof of the proposition, if the foliation $\T$ is such that $\degO
  L^\vee\leq0$,  then necessarily $\tilde\beta$ and $\tilde\gamma$ vanish
  identically (i.e the harmonic map $f:\tilde Y\pfd\Y_{p,q}$ is constant along the lifted leaves
  on $\tilde Y$) and $\degO\tilde W=0$. Observe however that if $\degO L^\vee<0$, it follows from a
  result of Bogomolov and McQuillan~\cite{BMQ} and from the stability for holomorphic foliations,
  see~\cite{Per} for instance, that the foliation is a fibration whose fibers are rational curves
  (at least if $M$ is projective). But then, the harmonic map is constant in the fibers of the
  fibration (by pluriharmonicity and the fact that a harmonic map on $\P^1_\C$ with values in a non
  positively curved manifold is constant) so that the result was known a priori.  
\end{rema}

\begin{rema}\label{closedleaves3} 
{\em Closed leaves III: Milnor-Wood inequality.} In a similar spirit, the convergence of currents
alluded to in Remark~\ref{closedleaves2} allows to give another proof of the Milnor-Wood inequality
on compact 
complex hyperbolic manifolds given by arithmetic lattices of type I by
deducing it from the inequality on the (infinitely many) totally geodesic curves they
  contain. It seems however difficult to build on this  
idea to infer the rigidity of maximal representations in this special case, while the more
general approach presented here will prove more fruitful. 
\end{rema}

\subsubsection{The singular locus of the Higgs field}\label{sec:sat}\hfill

\smallskip

We now study the equality case in Proposition~\ref{prop:MW}. Using the weak polystability of
  the Higgs bundle along the leaves of the tautological foliation $\T$ and the results of
  \textsection \ref{sec:closure} on the dynamics of $\T$, we
  show that if equality holds, then a component of the Higgs field, $\b$ or $\g$, is regular (in a
  sense to be defined below) on an 
  everywhere dense subset of $X$. 

\medskip

Say that a point $\xi\in\PTX$ is a {\em $\tilde\b$-regular} point, or that $\tilde\b$ is regular at
  $\xi$, if the rank of $\tilde \b_\xi:\tilde
W_\xi\otimes L_\xi\pfd\tilde V_\xi$ is the generic rank of $\tilde \b:\tilde W\otimes L\pfd\tilde
V$. Say that a point $x\in X$ is a {\em $\b$-regular}  point, or that $\b$ is regular at $x$, if the
fiber of $\PTX$ above $x$ consists only of $\tilde\b$-regular points. Points in $X$,
resp. $\PTX$, which are not
$\b$-regular, resp. $\tilde\b$-regular, are {\em $\b$-singular}, resp. {\em $\tilde\b$-singular}.  

Define accordingly the {\em singular locus} $\S_{\tilde \b}$ of $\tilde\b:\tilde
  W\otimes L\pfd\tilde V$ as the subset of $\tilde \b$-singular points in $\PTX$: 
$$
\S_{\tilde\b}=\{\xi\in\PTX\,\mid\,\rk\,(\tilde\b_\xi)<\rk\,\tilde\b\}
$$   
and the {\em singular locus} $\S_\b$ of $\b:W\otimes T_X\pfd V$ as the subset of $\b$-singular
points in $X$. Note that $\S_\b$ is by definition the projection of $\S_{\tilde\b}$ to $X$:
$$
\S_\b=\pi(\S_{\tilde\b})=\{x\in X\,\mid\,\exists\,\xi\in T_{X,x},\,\xi\neq 0,\mbox{ such that }\rk\,\b_x(\xi)<\rk\,\tilde\b\}.
$$  

One defines similarly $\tilde \g$- and $\g$- regular and singular points as well as the singular loci
  $\S_{\tilde\g}$ and $\S_\g$. 

\smallskip

Observe that while $\S_{\tilde\b}$ and $\S_{\tilde\g}$ are proper analytic subsets of $\PTX$,
$\S_\b$ and $\S_\g$ might well be the
whole $X$.

\medskip

\begin{lemma}\label{lem:sat}
If $\dego\tilde W=\rk\,\tilde\b\,\,\frac{\dego L^\vee}2$, the singular locus $\S_{\tilde\b}$ of
$\tilde\b:\tilde W\otimes L\pfd\tilde V$ is a proper
$\T$-saturated subset of $\PTX$.  If
  $\dego\tilde W=- \rk\,\tilde\g\,\,\frac{\dego L^\vee}2$, then the singular locus $\S_{\tilde \g}$ 
  of $\tilde\g:\tilde V\otimes L\pfd\tilde W$ is a proper $\T$-saturated subset of $\PTX$.
\end{lemma}

\begin{proof} We prove the assertion on $\tilde\b$. Call $r$ the generic rank of $\tilde
    \b$. Let $\cal N$ and $\cal I$ be  
respectively the kernel sheaf and the saturation of the image sheaf of 
$\tilde\beta:\tilde W\otimes L\rightarrow \tilde V$ and let $\S(\cal N)$ and $\S(\cal I)$ be their
singular loci (as defined just before Proposition~\ref{dec}). Observe that by definition, outside
of $\S_{\tilde \b}$, the rank of $\tilde\b_\xi$ is 
constant equal to $r$, so that $\cal N$, resp. $\cal I$, is the sheaf of sections of a subbundle of
$\tilde W\otimes L$, resp. $\tilde V$. This implies that $\S(\cal N)$ and $\S(\cal I)$ are included
in $\S_{\tilde\b}$.   

By the proof of Proposition~\ref{prop:MW},
if $\dego\tilde W=r\,\frac{\dego L^\vee}2$ then the foliated degrees of $\cal N\otimes L^\vee$ and 
$\cal I\oplus\tilde W$, which are leafwise Higgs subsheaves of 
$\tilde E$, vanish. Since $\tilde W/(\cal N\otimes L^\vee)$ and $\tilde V/\cal I$ are torsion
  free, by the weak polystability property of $(\tilde E,\tilde\t)$ proved in
Proposition~\ref{dec}~(\ref{sat}), $\S(\cal N)$ and $\S(\cal I)$ are both $\T$-saturated. Moreover there exist a
rank $q-r$ holomorphic subbundle $N$ 
of $\tilde W$, a rank $r$ holomorphic subbundle $I$ of $\tilde V$, both defined outside of 
the codimension at least 2 subset $\singsat:=\S(\cal N)\cup\S(\cal I)$ of $\PTX$, 
such that on $\PTX\backslash \singsat$, $\cal N\otimes L^\vee$ and $\cal I$ are the
sheaves of sections of $N$ and $I$. 

Since, outside of $\singsat$, $\tilde\beta$ maps $(\tilde W/N)\otimes L$ to $I$ and $\rk\, I=r=\rk\,
(\tilde W/N)$, 
the set of points $\xi\in\PTX\backslash\singsat$ where $\tilde\beta_\xi$ is not of rank $r$ is locally
given by the 
vanishing of a single holomorphic function and hence has codimension 1 if not empty. This means that
the components of 
$\S_{\tilde \b}$ of higher codimension  
are included in $\singsat$ and hence that $\tilde \beta:\tilde W\otimes L\pfd\tilde V$ has rank $r$, as
a vector bundle map, outside $\singsat\cup |\Delta|$, where $|\Delta|$ is the (possibly empty) 
divisorial part of $\S_{\tilde\b}$, i.e. the union of its irreducible components $\Delta_j$ of codimension
1. Thus $\S_{\tilde\b}$ is included in $\singsat\cup |\Delta|$, so that in fact, by our
  first observation, $\S_{\tilde\b}=\singsat\cup |\Delta|=\S(\cal N)\cup\S(\cal I)\cup|\Delta|$.

\smallskip

By an argument similar to~\cite{K}*{Chap.~V (8.5) p. 180},
there is a line bundle $[\Delta]$ on $\PTX$
corresponding to a divisor $\Delta=\sum_j a_j \Delta_j$ whose support is $|\Delta|$ (i.e. $a_j\geq
1$ for all $j$) such that $\det\cal I\simeq \det (\Im\tilde\b)\otimes[\Delta]$ on $\PTX$. Again by
the proof of Proposition~\ref{prop:MW}, $\dego\tilde W=r\,\frac{\dego L^\vee}2$ implies $\dego
\Im\tilde\b=\dego\cal I$, thus 
$\dego [\Delta]=\sum_ja_j\,\int_{\Delta_j}\Omegag=0$. This  
means that for all $j$, and at each smooth point $x$ of $\Delta_j$, the leaf $\L_x$ of $\T$ through
  $x$ is tangent to $\Delta_j$.
As the foliation is smooth, $\L_x$ must be contained in $\Delta_j$.
Now in $\Delta_j$, the smooth points are dense and the set of points whose leaves 
stay in $\Delta_j$ is closed, for it is analytic 
as explained in  the proof of (\ref{polystab}) in Proposition~\ref{dec}.  Thus $\Delta_j$ is
$\T$-saturated for all $j$.  
\end{proof}

\begin{rema}
As it is clear from its proof, Lemma~\ref{lem:sat} also holds more generally in the setting of
\textsection \ref{higgsfol} if Proposition~\ref{dec} applies. 
\end{rema}

\medskip

If we now consider the singular locus of $\b$ or $\g$ in $X$, 
Proposition~\ref{prop:dyn} implies immediately:

\begin{coro}\label{cor:small}
If $\dego\tilde W=\rk\,\tilde\b\,\,\frac{\dego L^\vee}2$, the singularity set $\S_\b$
of $\b:W\otimes T_X\pfd V$ is a proper analytic subset of $X$. If $\dego\tilde W=-
\rk\,\tilde\g\,\,\frac{\dego L^\vee}2$ then the singular locus $\S_\g$ 
  of $\g:V\otimes T_X\pfd W$ is a proper analytic subset of $X$.
\end{coro}

\begin{proof}
  We prove the assertion on $\b$.  Since $\S_{\tilde\b}$ is a proper closed subset of $\PTX$
  and is $\T$-saturated by Lemma~\ref{lem:sat}, 
 Proposition~\ref{prop:dyn} implies 
 that $\S_\b=\pi(\S_{\tilde\b})$ is a proper subset of $X$. Now $\S_{\tilde \b}$ is an analytic subset and
 $\pi$ a proper map, so
 $\S_\b$ is also an analytic subset of $X$.
\end{proof}

\subsubsection{Conclusion}\label{sec:concl}\hfill

\smallskip

We are now in position to conclude the proof of Theorem~\ref{thm:maxreduc}. So we assume that
  the reductive  
representation $\r$ is maximal. We want to prove that the $\r$-equivariant harmonic map $f$ is
  holomorphic or antiholomorphic, i.e. that one of the components of the Higgs field it defines
  vanishes. By the previous paragraph, we already know that one component is regular outside a
  proper analytic
  subset of $X$. The idea is that if $n\geq 2$, the 
  integrability property $[\t,\t]=0$ of the Higgs field forces the other component to
  vanish outside of this subset, hence everywhere.

\medskip 

Suppose that $\tau(\rho)>0$, so that
$\dego\tilde W=q\,\frac{\dego L^\vee}2$. We know from Proposition~\ref{prop:MW} that
$\rk\,\tilde\b=q$, hence from~\textsection \ref{sec:sat} that the set of $\b$-regular points 
$$
X\backslash\S_\b=\{x\in X\,\mid\, \b_x(\xi):W_x\fd V_x \mbox{ is injective for all }\xi\neq
0\mbox{ in }T_{X,x} \}
$$
is everywhere dense in $X$. 

\medskip

Let us fix a $\b$-regular point $x\in X$, i.e. $x\notin\S_\b$. For $\xi\neq 0$ in $T_{X,x}$, call
$I_\xi\subset V_x$ the 
image of $\beta_x(\xi):W_x\pfd V_x$ (which is injective), and $I_\xi^\perp$ its
orthogonal complement in $V_x$ w.r.t. the harmonic metric. Observe that $I_\xi^\perp$ is also the
  orthogonal complement of $I_\xi\oplus W_x$ in $E_x=V_x\oplus W_x$, because $V_x$ and $W_x$ are orthogonal for
  the harmonic metric.

\smallskip

Using the integrability property of the Higgs field and again the weak polystability along the
leaves, we first prove

\begin{lemma}\label{lem:vanish}
For all $\eta$ and all $\xi\neq 0$ in $T_{X,x}$, $\g_x(\eta)$ vanishes on $I_\xi^\perp$.     
\end{lemma}
\begin{proof}
Since $\dego\tilde W=q\,\frac{\dego L^\vee}2$, we know from the proof of
  Proposition~\ref{prop:MW} that the kernel sheaf $\cal N$ of $\tilde\b:\tilde W\otimes L\pfd\tilde
  V$ is zero and that its image 
  sheaf $\cal I$ satisfies $\dego (\tilde W\oplus\cal
  I)=0$. By weak polystability along the leaves, see
  Proposition~\ref{dec}, outside 
  the singular locus $\S_{\tilde\b}$ of $\tilde\b$, there is a subbundle $I$ of $\tilde V$ such
  that $\cal I$ is the sheaf of sections of $I$ and $(\tilde E,\t)=(\tilde W\oplus I,\t_{|W\oplus I})\oplus(I^\perp,\t_{|I^\perp})$ is a Higgs bundle
  decomposition along the leaves of $\T$, where $I^\perp$ is the orthogonal complement in $\tilde E$
  of $\tilde W\oplus I$ w.r.t. the lifted harmonic metric on $\tilde E$. Since $x\notin\S_\b$, this
  means that for all $\xi\neq 0$ 
  in $T_{X,x}$, $W_x\oplus I_\xi$ and $I_\xi^\perp$ are invariant by $\t_x(\xi)$.    

But $I_\xi^\perp\subset V_x$ and $\t_x(\xi)_{|V_x}=\g_x(\xi)$ maps $V_x$ to $W_x$.
Hence $I_\xi^\perp\subset \Ker\g_x(\xi)$.

On the other hand, the integrability property $[\t,\t]=0$ of the Higgs field means that
$\b_x(\xi)\circ\g_x(\eta)=\b_x(\eta)\circ\g_x(\xi)$ for all $\xi,\eta\in T_{X,x}$. Therefore, if
$v\in\Ker\g_x(\xi)$ for some $\xi\neq 0$, then for all $\eta$ we have 
$\b_x(\xi)(\g_x(\eta)v)=\b_x(\eta)(\g_x(\xi)v)=0$ and hence $\g_x(\eta)v=0$ since $\b_x(\xi)$ 
is injective. Hence for all $\xi\neq 0$ and all $\eta$, $\Ker\g_x(\xi)\subset\Ker\g_x(\eta)$. 
\end{proof}

The next lemma shows that the subspaces $I_\xi^\perp$ for $\xi\neq 0$ generate $V_x$. This is
  the only point in the proof for which the assumption that $n\geq 2$ is required.

\begin{lemma}\label{lem:intersect}
Assume that $n\geq 2$. Then $\cap_{\xi\neq 0}I_\xi=\{0\}$.    
 \end{lemma}

 \begin{proof}
Let indeed $v$ be in $\cap_{\xi\neq 0}I_\xi$. 
For all $\xi\neq 0$, there exists $\varphi(\xi)\in W_x$ such that
  $\b_x(\xi)\varphi(\xi)=v$. By the injectivity of $\b_x(\xi)$, $\varphi(\xi)$ is unique and
  $\varphi$ is a well-defined map from $T_{X,x}\backslash\{0\}$ to $W_x$. Since $\varphi$ is
  locally given by inverting a $q$-by-$q$ submatrix of $\b_x(\xi)$, it is holomorphic on
  $T_{X,x}\backslash\{0\}$.
Because 
$n\geq 2$, the map $\varphi$ can be extended holomorphically to $0\in T_{X,x}$ and necessarily
$\b_x(0)\varphi(0)=v$ so that 
$v=0$ since $\b_x:T_{X,x}\pfd\Hom(W_x,V_x)$ is linear. 
\end{proof}
 
Together these lemmas imply that for $n\geq 2$, $\g$ vanishes outside a proper analytic subset of $X$,
hence everywhere and the
$\rho$-equivariant harmonic map 
$f$ is holomorphic. 

\medskip

In the same manner, if $\tau(\r)<0$, then $\dego \tilde W=-\frac q2\,\dego L^\vee$, and if $n\geq 2$,
$\b$ vanishes
outside the singular locus $\S_\g$ of $\g$, hence
identically, and the 
$\rho$-equivariant harmonic map $f$ is antiholomorphic.

\subsection{Proof of Theorem~\ref{main} and Corollary~\ref{description} for reductive
  representations}\label{sec:end}\hfill
 
\smallskip

Recall from \textsection \ref{sec:Ypq} that the maximal value of the holomorphic sectional 
curvature of the $\SU(p,q)$-invariant metric $\o_{\Y_{p,q}}$ of $\Y_{p,q}$ is $-1/q$.  
The Ahlfors-Schwarz-Pick lemma~\cite{Ro80} therefore implies that if $f:\H\pfd\Y_{p,q}$ is
holomorphic, then $f^\star\o_{\Y_{p,q}}\leq q\,\o$. Moreover, this inequality is an equality only if
the induced holomorphic sectional curvature on the image of $f$ is everywhere maximal, i.e. equal to
$-1/q$, and we proved the following result in~\cite{KozMobRank2}*{\textsection 3.1}:  

\begin{prop}\label{special}
Let $f:\H\pfd\Y_{p,q}$ be a holomorphic map such that $f^\star\o_{\Y_{p,q}}= q\,\o$.  Then $p\geq
qn$ and up to the composition of $f$ by an isometry of  
$\Y_{p,q}$, $f$ is equal to the following holomorphic totally geodesic embedding:
\begin{equation*}
f_{\rm diag}:\H\ni z=\left(\begin{array}{c}z_1\\ z_2\\ \vdots \\ z_n
        \end{array}\right)\longmapsto
Z=
\left(\begin{array}{cccc}
z & 0 & \cdots & 0 \\
0 & z & & \vdots \\
\vdots & & \ddots & 0\\
0 & \cdots & 0 & z \\
0 & 0 & 0 & 0 \\
\end{array}\right)
\in\Y_{p,q}.\,
\end{equation*}
The totally geodesic map $f_{\rm diag}$ is equivariant with respect to the standard diagonal
embedding $\r_{\rm diag}:\SU(n,1)\hookrightarrow \SU(n,1)^q \hookrightarrow \SU(nq,q)
\hookrightarrow\SU(p,q)$. The stabilizer of its image in $\Y_{p,q}$ is an almost-direct product of
$\r_{\rm diag}(\SU(n,1))$  with its centralizer $K$ in $\SU(p,q)$, which is compact and acts
trivially on $f_{\rm diag}(\H)$. 
\end{prop}

This proposition shows that Theorem~\ref{main} and Corollary~\ref{description} for reductive
representations in $\SU(p,q)$ are direct consequences of Theorem~\ref{thm:maxreduc}. Indeed, this
theorem says that if $\r$ is a reductive maximal representation in $\SU(p,q)$ of a torsion free
uniform lattice $\G$ of $\SU(n,1)$, $n\geq 2$, then the $\r$-equivariant harmonic map
  $f:\H\pfd\Y_{p,q}$ is holomorphic or antiholomorphic.

  If $f$ is
holomorphic then the Ahlfors-Schwarz-Pick lemma gives the pointwise inequality
$f^\star\o_{\Y_{p,q}}\leq q\,\o$ whereas the maximality of $\rho$ means that $\int_X
f^\star\o_{\Y_{p,q}}\wedge\o^{n-1}=q\int_X\o^n$, so that necessarily
$f^\star\o_{\Y_{p,q}}= q\,\o$. By Proposition~\ref{special} we have $p\geq qn$ (which proves
Theorem~\ref{main})  and up to composition by an
element of $\SU(p,q)$, $f=f_{\rm diag}$, which is the first assertion of
  Corollary~\ref{description}. The second assertion follows easily (to prove that $\r$ is faithful,
  note that $\G$ being torsion free, it is isomorphic to its projection to $\PU(n,1)$ which acts
  effectively on $\H$). The third assertion follows from our description of the stabilizer of
  $f_{\rm diag}(\H)$ and the fact that up to conjugacy we may assume that
  $f_{\rm diag}$ is $\r$-equivariant so that for $\g\in\G$,  $\r(\g)$ acts on $f_{\rm diag}(\H)$ as
  $\r_{\rm diag}(\g)$.

If $f$ is antiholomorphic then 
the maximality of $\r$ implies $f^\star\o_{\Y_{p,q}}= - q\,\o$, so that again $p\geq qn$ and essentially
$f=\bar f_{\rm diag}$. 

\subsection{Non reductive representations}\label{nonreductive} \hfill

\smallskip

A very general result of M.~Burger, A.~Iozzi and A.~Wienhard asserts that so-called {\em tight}
  representations of lattices, uniform or not, 
of $\SU(n,1)$, $n\geq 1$, in Hermitian Lie groups are always reductive, see~\cite{BIW09}*{Corollary
  4}. Maximal representations are tight, so that the results of \textsection \ref{sec:reduc} and
\textsection \ref{sec:end} for reductive representations imply Theorem~\ref{main} and
Corollary~\ref{description} in the general case. 

It is however interesting to see that one can deduce the inexistence of non reductive maximal
representations in $\SU(p,q)$ from the rigidity just established for reductive maximal ones. 
We  explain here how to do this by deforming non reductive representations to reductive ones, a known 
operation sometimes called {\em  semi-simplification}:

\begin{lemma}\label{lemma:semisimple}
  A non reductive group homomorphism $\r$ of a group $\G$ in a semisimple Lie group $G$ without compact
  factors can be deformed to a reductive homomorphism $\r_{ss}:\G\pfd P$, where $P$ is a proper
  parabolic subgroup of $G$.
\end{lemma}

Together with our results for reductive representations, this implies:

\begin{coro}\label{cor:maxreduc}
  Let $\r$ be a non reductive representation of a torsion free uniform lattice $\G$ of $\SU(n,1)$ in
  $\SU(p,q)$, $p\geq q\geq 1$. Then $\r$ satisfies the Milnor-Wood inequality $|\tau(\r)|\leq
  q\,\vol(X)$. Moreover,  if $n\geq 2$, $\r$ is not maximal.   
\end{coro}

Therefore the proofs of Theorem~\ref{main} and Corollary~\ref{description} for
  representations in $\SU(p,q)$ are complete.

\begin{proof}[Proof of Lemma~\ref{lemma:semisimple}] This follows
    e.g. from~\cite{Richardson}. Let $\rho:\G\pfd G$ a non reductive
  homomorphism: the Zariski
closure $\overline{\r(\G)}^{\,\textsf z}$ of $\r(\G)$ in $G$  is not a reductive group, so that its
unipotent radical $U$ is not trivial. Let $L$ be a Levi factor of $\overline{\r(\G)}^{\,\textsf
  z}$. By~\cite{Richardson}*{Proposition~2.6}, there exists a 1-parameter subgroup $\lambda$ of $G$,
such that $\overline{\r(\G)}^{\,\textsf z}$ is contained in the parabolic subgroup
$P(\lambda):=\{g\in G\,|\,\lim_{t\rightarrow +\infty}\lambda(-t)\,g\,\lambda(t)\mbox{ exists}\}$,
$U$ is contained in the unipotent radical $N(\lambda):=\{g\in G\,|\,\lim_{t\rightarrow
  +\infty}\lambda(-t)\,g\,\lambda(t)=1\}$ of $P(\lambda)$, and $L$ is contained in
$L(\lambda):=\{g\in G\,|\,\lim_{t\rightarrow +\infty}\lambda(-t)\,g\,\lambda(t)=g\}$ which is a Levi
subgroup of $P(\lambda)$. 

The homomorphism $\r_{ss}$ is then defined by 
$\rho_{ss}(\g)= \lim_{t\rightarrow
  +\infty}\lambda(-t)\,\rho(\g)\,\lambda(t)\in L(\lambda)$ for all $\g\in\G$. It is reductive and
maps $\G$ to $L\subset L(\lambda)\subset P(\lambda)$. The
parabolic subgroup 
$P(\lambda)$ is indeed a proper subgroup of $G$ since its unipotent radical contains $U$. 
\end{proof}

\begin{proof}[Proof of Corollary~\ref{cor:maxreduc}]
Deform $\r$
to the reductive representation $\r_{ss}$ as in Lemma~\ref{lemma:semisimple}. 
The
representation $\rho_{ss}$ belongs to the connected component of $\rho$ in the space 
$\Hom(\G,\SU(p,q))$, and 
therefore $\tau(\r_{ss})=\tau(\r)$.  
Proposition~\ref{prop:MW} gives the 
Milnor-Wood inequality on $\tau(\r_{ss})$, hence on $\tau(\r)$.  

Assume moreover that $n\geq 2$ and that $\rho$ is maximal.
Hence so is $\rho_{ss}$, and
by \textsection \ref{sec:end} we know that   
$p\geq nq$ and that, up to conjugacy by an element of $\SU(p,q)$, $\r_{ss}$ is a product
  $\r_{\rm cpt}\times\r_{\rm diag}$. Moreover  
$\r_{ss}(\G)$ is a lattice in the stabilizer of $\r_{\rm diag}(\H)$, which by
Proposition~\ref{special} is an almost direct product of the simple noncompact 
  group $\r_{\rm
    diag}(\SU(n,1))$ with its compact centralizer $K$. Therefore by ~\cite{Dani}*{Corollary~4.2},
  the Zariski closure $\overline{\r_{ss}(\G)}^{\,\textsf z}$ of $\rho_{ss}(\G)$ must contain $\r_{\rm 
    diag}(\SU(n,1))$, so that the 
centralizer of $\overline{\r_{ss}(\G)}^{\,\textsf z}$ in $\SU(p,q)$ is included in $K$. This is a
contradiction since by   
Lemma~\ref{lemma:semisimple}, $\overline{\r_{ss}(\G)}^{\,\textsf z}$ sits in a Levi subgroup of a
proper parabolic subgroup of $G$, hence its centralizer is not compact.  
\end{proof}

\section{Representations in classical Hermitian Lie groups other than $\SU(p,q)$}\label{sec:others}\hfill

In order to conclude the proof of Theorem~\ref{main},  one needs to rule out possible maximal
  representations in the remaining classical Hermitian Lie groups, namely $\SO_0(p,2)$ with $p\geq 3$, 
$\Sp(m,\R)$ with $m\geq 2$, and $\SO^\star(2m)$ with $m\geq 4$. We recall that their associated
symmetric spaces' ranks are $2$, $m$ and $\lfloor m/2\rfloor$ respectively. 

This section is therefore devoted to the proof of the following: 

\begin{theo}\label{others}
  There are no maximal representations of a uniform (torsion free) lattice of $\SU(n,1)$, $n\geq 2$,
  in the classical Hermitian Lie groups  $\SO_0(p,2)$ with $p\geq 3$, 
$\Sp(m,\R)$ with $m\geq 2$, and $\SO^\star(2m)$ with $m\geq 4$.
\end{theo}

Representations in $\SO_0(p,2)$ can be dealt with using the results of~\cite{KozMobRank2},
see~\textsection \ref{sec:SOp2}. Representations in $\Sp(m,\R)$ and $\SO^\star(2m)$ will
be treated in \textsection \ref{sec:SpmR} and \textsection \ref{sec:SOstar}, respectively. In these
two latter
cases the proof relies on the results of \textsection \ref{max}. Indeed, representations in these
groups can be seen as representations in 
$\SU(m,m)$. As we shall see, in the case of a representation into $\Sp(m,\R)$ or into
$\SO^\star(2m)$ for $m$ even, 
easy curvature computations show that the Milnor-Wood inequality the representation should satisfy
is the same as the Milnor-Wood inequality it satisfies (by \textsection \ref{max}) when seen as a
representation in $\SU(m,m)$. A maximal representation in these groups would therefore be a particular
maximal representation in $\SU(m,m)$,  and this is impossible, again by \textsection
\ref{max}. For representation in $\SO^\star(2m)$ with $m$ odd however, and although the general idea
is the same as in the $\SU(p,q)$ case, more work is needed, including rather painful verifications.

\medskip

\begin{rema} 
We assume $p\geq 3$ for $\SO_0(p,2)$ because $\SO_0(2,2)$ is not simple (it is locally isomorphic to
$\SU(1,1)\times\SU(1,1)$), $m\geq 2$ for $\Sp(m,\R)$ because $\Sp(1,\R)$ is isomorphic to $\SU(1,1)$,
and $m\geq 4$ for $\SO^\star(2m)$ because for $m=2$, $\SO^\star(4)$ is not simple (it is
locally isomorphic to $\SU(1,1)\times\SU(1,1)$), whereas for $m=3$, $\SO^\star(6)$ is locally
isomorphic to $\SU(3,1)$ (and therefore in this last case there are maximal representations of
lattices of $\SU(2,1)$
and $\SU(3,1)$ in this group). Note that $\SO^\star(8)$ is locally isomorphic to
$\SO_0(6,2)$.
\end{rema}  

\begin{rema}
  Since there are no maximal (in our sense) representations of a uniform lattice $\G$ of $\SU(n,1)$,
  $n\geq 2$, in any of the groups $\SO_0(p,2)$ with
  $p\geq 3$,  
$\Sp(m,\R)$ with $m\geq 2$, $\SO^\star(2m)$ with $m\geq 4$, or $\SU(p,q)$ with $p\geq q\geq 1$ but
$p<qn$, it is natural to wonder what is the maximal possible value of the Toledo invariant of a
representation 
$\r:\G\pfd G$, for $G$ a specific group in this list, and whether a representation realizing this
maximal value has some nice geometric properties. This seems to be a difficult question, whose answer
probably depends heavily on the specific target Lie group $G$. 

Observe however that the Milnor-Wood inequality is satisfied by representations of
surface groups, i.e. uniform lattices in $\SU(1,1)$, and that in this case maximal representations
(as defined in this paper) exist in any Hermitian Lie group. 
\end{rema}

\subsection{Representations in $\SO_0(p,2)$, $p\geq 3$}\label{sec:SOp2}\hfill

\smallskip

The case of representations in the groups $\SO_0(p,2)$ ($p\geq 3$) has already been treated
in~\cite{KozMobRank2} where it was shown that such representations satisfy the 
inequality $|\tau(\rho)|\leq \frac{n+1}n\vol(X)$. This is stronger than the Milnor-Wood inequality
since the rank of  the symmetric 
space associated to $\SO_0(p,2)$ is $2$ and $n\geq 2$. Hence there are no maximal representations in
this case.

\subsection{Representations in $\Sp(m,\R)$, $m\geq 2$}\label{sec:SpmR}\hfill

\smallskip

This group may be described as the following subgroup of $\SU(m,m)$: 
$$
\Sp(m,\R)=\{g\in\SU(m,m)\,|\,g^\top J_{m,m}\,g=J_{m,m}\}\
$$
where $g^\top$ is the transpose of the matrix $g$ and $J_{m,m}$ is the $2m$-by-$2m$ matrix
$$
J_{m,m}=\left(
\begin{array}{cc}
0 & 1_m \\
-1_m & 0 \\
\end{array}
\right).
$$
The associated symmetric space $\Y$ is totally geodesically, holomorphically, and
  $\Sp(m,\R)$-equivariantly, 
embedded in the symmetric space $\Y_{m,m}$ 
associated to $\SU(m,m)$ as 
$$\Y=\{Z\in M_{m}(\C)\,|\,1_m-Z^\star Z>0\ {\rm and}\ Z^\top =Z\}\,
\subset\{Z\in M_{m}(\C)\,|\,1_m-Z^\star Z>0\}=\Y_{m,m}
$$

Let us call $\iota:\Y\pfd\Y_{m,m}$ this embedding, $\o_\Y$ the $\Sp(m,\R)$-invariant metric on $\Y$
and $\o_{\Y_{m,m}}$ the $\SU(m,m)$-invariant metric of $\Y_{m,m}$, both normalized to that the
minimum of their holomorphic sectional curvature is $-1$. 

\begin{lemma}\label{curv:SpmR}
We have $\iota^\star \o_{\Y_{m,m}}=\o_\Y$. 
\end{lemma}

\begin{proof} 
Both metrics are $\Sp(m,\R)$-invariant metrics on $\Y$, it is therefore enough to show that their
normalizations agree, namely that the minimum of the holomorphic sectional curvature of $\iota^\star
\o_{\Y_{m,m}}$ is $-1$.   
  Because $\iota$ is totally geodesic, the holomorphic sectional curvature of $\iota^\star \o_{\Y_{m,m}}$ is the
  restriction of the holomorphic sectional curvature of the metric $\o_{\Y_{m,m}}$ to complex lines
  in $T_\Y$. Now, at a point
  $o\in\Y\subset\Y_{m,m}$, the holomorphic tangent space $T_{\Y,o}$ to $\Y$ identifies with the
  subspace $S_m(\C)$ of symmetric matrices in $M_m(\C)\simeq T_{\Y_{m,m},o}$. Therefore, by the
  formula of \textsection \ref{sec:prelim}, the 
  holomorphic sectional curvature of $\iota^\star \o_{\Y_{m,m}}$ on the complex line $\la A \ra$
  generated by a nonzero symmetric $A\in M_m(\C)$ is $-\tr((A^\star A)^2)/(\tr (A^\star A))^2$ so
  that its minimum value is indeed $-1$ (attained for example by a diagonal matrix with only one non
  zero entry equal to $1$).
\end{proof}

Let $\rho$ be a representation of a lattice $\G$ of $\SU(n,1)$ in $\Sp(m,\R)$, and let
$\r':\G\pfd\SU(m,m)$ be $\r$ composed with the inclusion $\Sp(m,\R)\subset\SU(m,m)$.  By the
very definition of the Toledo invariant, we have 
$$\tau(\r)=\frac 1{n!}\,\int_Xf^\star\o_\Y\wedge \o^{n-1}=\frac 1{n!}\,\int_X f^\star\iota^\star
\o_{\Y_{m,m}}\wedge\o^{n-1}=\frac 1{n!}\,\int_X (\iota\circ f)^\star \o_{\Y_{m,m}}\wedge\o^{n-1}=\tau(\r')$$ 
which means that the Toledo invariant of $\rho:\G\pfd\Sp(m,\R)$ is the same as the Toledo invariant
of $\rho':\G\pfd\SU(m,m)$. 
Since the ranks of $\Y$ and $\Y_{m,m}$ are both equal to $m$, the results of~\textsection
\ref{max} give the Milnor-Wood inequality $|\tau(\rho)|\leq m\,\vol(X)$. Moreover, if the representation
$\rho$ in $\Sp(m,\R)$ is maximal, then the representation $\r'$ in $\SU(m,m)$ is also maximal. But
there are no such representations since by \textsection \ref{max},  maximal representations in $\SU(p,q)$ exist only if $p\geq
nq$ (as always, we assume that $n\geq 2$).

\subsection{Representations in $\SO^\star(2m)$, $m\geq 4$}\label{sec:SOstar} \hfill

\smallskip

We proceed as in the previous paragraph by considering representations with values in
$\SO^\star(2m)$ as special representations with values in $\SU(m,m)$. For $m$ even, this allows to
conclude as quickly as in the $\Sp(m,\R)$ case. 

\medskip

The group $\SO^\star(2m)$ may be described as the following subgroup of $\SU(m,m)$:
$$
\SO^\star(2m)=\{g\in\SU(m,m)\,|\,g^\top J'_{m,m}\,g=J'_{m,m}\}
$$
where $J'_{m,m}$ is the $2m$-by-$2m$ matrix
$$
J'_{m,m}=\left(
\begin{array}{cc}
0 & 1_m \\
1_m & 0 \\
\end{array}
\right).
$$
Observe that if $q'_{m,m}$ is the quadratic form on $\C^{2m}$ whose matrix in the canonical
  basis is $J'_{m,m}$ then $\SO^\star(2m)$ is a subgroup of
  $\SO(2m,\C)=\SO(\C^{2m},q'_{m,m})$. In fact it is a real form of this complex group.

The associated symmetric space $\Y$ is totally geodesically, holomorphically, and
  $\SO^\star(2m)$-equivariantly, embedded in $\Y_{m,m}$ as 
$$ 
\Y=\{Z\in M_{m}(\C)\,|\,1_m-Z^\star Z>0\ {\rm and}\ Z^\top =-Z\}\,
\subset\{Z\in M_{m}(\C)\,|\,1_m-Z^\star Z>0\}=\Y_{m,m}
$$

Again, call $\iota:\Y\pfd\Y_{m,m}$ this embedding, $\o_\Y$ the $\SO^\star(2m)$-invariant metric on $\Y$
and $\o_{\Y_{m,m}}$ the $\SU(m,m)$-invariant metric of $\Y_{m,m}$, both normalized to that the
minimum of their holomorphic sectional curvature is $-1$. 

\begin{lemma}\label{curv:SOstar}
We have $\iota^\star \o_{\Y_{m,m}}=2\,\o_\Y$. 
\end{lemma}

\begin{proof} 
The proof is entirely similar to the proof of Lemma~\ref{curv:SpmR}: these two metrics on $\Y$ are
$\SO^\star(2m)$-invariant and all we need to prove is that the minimum of the holomorphic sectional
curvature of $\iota^\star \o_{\Y_{m,m}}$ is $-\frac 12$.  

At a point
  $o\in\Y\subset\Y_{m,m}$, the holomorphic tangent space $T_{\Y,o}$ to $\Y$ identifies with the
  subspace of 
  skew-symmetric matrices in $M_m(\C)\simeq T_{\Y_{m,m},o}$. Therefore, as in the proof of
  Lemma~\ref{curv:SpmR}, the
  holomorphic sectional curvature of $\iota^\star \o_{\Y_{m,m}}$ on the complex line $\la A \ra$
  generated by a nonzero skew-symmetric $A\in M_m(\C)$ is $-\tr((A^\star A)^2)/(\tr (A^\star A))^2$.

By Youla's
decomposition~\cite{You}, there exists a unitary matrix $U\in\U(m)$ such that $U^\top AU$ is a block
diagonal matrix with $\lfloor m/2\rfloor$ skew-symmetric 2-by-2 blocks $\left(\begin{matrix}0& -\alpha_i\\
    \alpha_i& 0\end{matrix}\right)$, where $\alpha_i\in\R$, and one additional zero on the diagonal
if $m$ is odd. Then, as $(U^\top AU)^\star (U^\top AU)=U^{-1}A^\star AU$, we get 
$$
\frac{\tr\bigl((A^\star A)^2\bigr)}{\bigl(\tr(A^\star A)\bigr)^2}=\frac{2\sum
  \alpha_i^4}{(2\sum\alpha_i^2)^2}
$$
which clearly implies the result.
\end{proof}

\medskip

Let $\rho$ be a representation of a lattice $\G$ of $\SU(n,1)$ in $\SO^\star(2m)$, and
$\r':\G\pfd\SU(m,m)$ be $\r$ composed with the inclusion $\SO^\star(2m)\subset\SU(m,m)$. By
definition, we have  
$$
\tau(\r)=\frac 1{n!}\,\int_Xf^\star\o_\Y\wedge \o^{n-1}=\frac 12 \,\frac 1{n!}\,\int_X (\iota\circ
f)^\star \o_{\Y_{m,m}}\wedge\o^{n-1}=\frac 12\,\tau(\r').
$$ 
As a consequence, the 
Milnor-Wood inequality $|\tau(\r)|\leq \lfloor m/2\rfloor\vol(X)$ is equivalent to the inequality
$|\tau(\rho')|\leq 2\lfloor m/2\rfloor\vol(X)$. 

\medskip

If $m$ is even, the Milnor-Wood inequality for 
$\r:\G\pfd\SO^\star(2m)$ is therefore the
usual Milnor-Wood inequality for $\r':\G\pfd\SU(m,m)$ and $\r$ is maximal if and only if $\r'$ is
maximal. As in the previous paragraph we may apply 
the results of \textsection \ref{max} to obtain the inexistence of
maximal representations in $\SO^\star(2m)$, $m$ even. 

\medskip

We assume from now on that $m$ is odd. 
Theorem~\ref{others} in this case is a consequence of the following two results:

\begin{prop}\label{prop:holSOstar}
Let $\r$ be a reductive maximal representation of a uniform lattice $\G\subset\SU(n,1)$ in
$\SO^\star(2m)$. Assume that $n\geq 2$ and that $m\geq 3$ is odd. Then if $\tau(\r)>0$,
resp. $\tau(\r)<0$, there exists a $\rho$-equivariant 
holomorphic, resp. antiholomorphic, map  $f:\H\pfd \Y$. 
\end{prop}

\begin{prop}\label{prop:maxsostar}
If $n\geq 2$ and $m\geq 4$, there are no holomorphic map
$f:\H\pfd \Y$ such that $f^\star\o_\Y=\lfloor m/2\rfloor\,\o$.
\end{prop}

Indeed, to prove Theorem~\ref{others}, assume that there is a maximal representation of a lattice
$\G$ of $\SU(n,1)$, $n\geq 2$, in $\SO^\star(2m)$, with $m$ odd and $m\geq 5$. Then we may either
apply~\cite{BIW09}*{Corollary 4} to get that this representation is reductive, or 
semisimplify this representation as in \textsection \ref{nonreductive} to obtain a reductive 
representation $\rho$ with the same Toledo invariant, so that $\r$ is again maximal. By
Proposition~\ref{prop:holSOstar}, if $\tau(\r)>0$,  there exists a $\r$-invariant holomorphic map $f:\H\pfd\Y$ and by
the  Ahlfors-Schwarz-Pick lemma \cite{Ro80}, $f^\star\o_\Y\leq   \lfloor m/2\rfloor\,\o$. Since $\r$
is maximal, necessarily   $f^\star\o_\Y=  \lfloor m/2\rfloor\,\o$ and this is a contradiction by
Proposition~\ref{prop:maxsostar}. If $\tau(\r)<0$ then there is a antiholomorphic map $f:\H\pfd\Y$. But in
this case the conjugate $\bar f$ is holomorphic and satisfies $\bar f^\star\o_\Y=  \lfloor
m/2\rfloor\,\o$, again a contradiction. 

\medskip

\begin{proof}[Proof of Proposition~\ref{prop:holSOstar}] 

We work with the Higgs bundle $(E,\t)$ associated to the reductive representation $\r$ in $\SO^\star(2m)$. As
in the $\SU(p,q)$-case, this Higgs bundle has a real structure. Since as we saw $\SO^\star(2m)$ is
a subgroup of $\SU(m,m)$, the Higgs bundle $(E,\t)$ is in particular a
$\SU(m,m)$-Higgs bundle, so that we have $(E,\t)= (V\oplus W,\b\oplus\g)$, with $\rk\, V=\rk\, W=m$,
$\b:W\otimes T_X\pfd V$ and $\g:V\otimes T_X\pfd W$. Because $\rho$ takes its values in
$\SO^\star(2m)$, which is a real form of $\SO(2m,\C) =\SO(\C^{2m}, q'_{m,m})$,  we have
moreover an identification of $V$ with $W^\star $, and for all $\xi\in 
T_X$, $\beta(\xi)\in \Hom(W,W^\star)$ and  $\g(\xi)\in \Hom(W^\star,W)$ are skew-symmetric,
  namely
for all $w_1,w_2\in W$, $(\b(\xi)w_1)(w_2)=-(\b(\xi)w_2)(w_1)$ and for all $v_1,v_2\in V=W^\star$,
$v_1(\g(\xi)v_2)=-v_2(\g(\xi)v_1)$. 

 The harmonic metric on $E=V\oplus W$ comes from a reduction of the structure group of $E$ to the maximal
compact subgroup $\U(n)$ of $\SO^\star(2n)$. Therefore it is also compatible with the real structure in
the sense that if $(w_1,\ldots,w_m)$ is an orthonormal basis of the fiber $W_x$ above some $x\in X$
then the dual basis 
$(w^\star_1,\ldots,w^\star_m)$ of $V_x=W_x^\star$ is also orthonormal. Equivalently, for all subspace $F$ of
$W_x$ or of $V_x$, we have 
$(F^\perp)^\circ=(F^\circ)^\perp$, where if $F$ is a subspace of $W_x$, resp. $V_x$,  $F^{\perp}$ is the
orthogonal complement of $F$ in $W_x$, resp. $V_x$, w.r.t. the harmonic metric, and $F^\circ=\{v\in
V_x=W^\star_x\,\mid\, v_{| F}=0\}\subset V_x$, resp. $F^\circ=\{w\in W_x\,\mid\,v(w)=0,\,\forall
v\in F\}\subset W_x$.      

\smallskip

As in \textsection \ref{max}, we lift the Higgs bundle over $X$ to the projectivized tangent bundle
$\PTX$ and the Milnor-Wood inequality $|\tau(\r)|\leq \lfloor \frac m2\rfloor\,\vol(X)$ is equivalent to
$|\dego \tilde W|\leq 2 \, \lfloor \frac m2\rfloor\,\frac{\dego L^\vee}{2}$ (the factor 2 comes from
Lemma~\ref{curv:SOstar}).  
    
Since $m$ is odd, we therefore need to prove that $|\dego \tilde W|\leq (m-1)\,\frac{\dego
  L^\vee}{2}$. Now, $\tilde\beta$ and $\tilde\gamma$ being skew-symmetric, their generic ranks are
bounded above by $m-1$ (again because $m$ is odd). Thus Proposition~\ref{prop:MW}  proves the
Milnor-Wood inequality.

\medskip

If the representation is maximal, say with $\dego \tilde W>0$, then $\dego \tilde W=
  (m-1)\,\frac{\dego L^\vee}{2}$ and the generic rank of
$\tilde\beta:\tilde W\otimes L\pfd\tilde V$ on $\PTX$ is $m-1$.  Moreover, 
by Corollary~\ref{cor:small}, the  
singular locus $\S_\b$ of $\b$ is a proper analytic subset of $X$.

\medskip

We again work above a single point $x\in X$, $x\notin\S_\b$.  If $\xi\in T_{X,x}$, we will
write $\b(\xi)$, resp. $\g(\xi)$, for $\b_x(\xi)$, resp. $\g_x(\xi)$.

If $\xi\neq 0$, we know since $x\notin\S_\b$ that $\b(\xi)$ has rank $m-1$. We write
$N_\xi$ for the 1-dimensional kernel of $\beta(\xi):W_x\pfd V_x$, and $I_\xi$ for its
$(m-1)$-dimensional image. We denote by $N_\xi^\perp\subset W_x$ and $I_\xi^\perp\subset V_x$ their
orthogonal complements 
w.r.t. the harmonic metric. We remark that by skew-symmetry of $\b(\xi)$, $I_\xi\subset {N_\xi}^\circ$ and
  that since $\rk\,\b(\xi)=m-1$, in fact $I_\xi={N_\xi}^\circ$.
 
\smallskip

 We want to proceed as for the $\SU(p,q)$ case in \textsection \ref{sec:concl}. We  
have the exact same statement as Lemma~\ref{lem:vanish}, although the proof is slightly different:

\begin{lemma}\label{lem:vanish2}
  For all $\eta$ and all $\xi\neq 0$ in $T_{X,x}$, $\g(\eta)$ vanishes on $I_\xi^\perp$ and
hence maps $V_x$ to $\cap_{\xi\neq 0}N_\xi^\perp$.  
\end{lemma}

\begin{proof}
Let $\xi\neq 0$.  Exactly as in Lemma~\ref{lem:vanish}, $\g(\xi)$ vanishes on $I_\xi^\perp\subset V_x$
by weak polystability along 
the leaves, because $I_\xi^\perp$ must be stable by the Higgs field.  

Hence, for all $\eta$, by integrability of the Higgs field,
$\b(\xi)\circ\g(\eta)=\b(\eta)\circ\g(\xi)$ vanishes on 
$I_\xi^\perp$ so that 
$\g(\eta)$ maps $I_\xi^\perp$ to $N_\xi$. But, again by weak polystability, $\g(\eta)$ also maps
$V_x$ to $N_\eta^\perp$, because $N_\eta^\perp\oplus I_\eta$ is stable by the Higgs field. Therefore
$\g(\eta)(I_\xi^\perp)\subset N_\xi\cap N_\eta^\perp$, so that for $\eta$ close to $\xi$, and
hence for all $\eta$,    $\g(\eta)(I_\xi^\perp)=0$. 

 Now, $\g(\eta)$ being skew-symmetric, $\Im \g(\eta)\subset (\Ker\g(\eta))^\circ$, so that 
$\Im \g(\eta)\subset (I_\xi^\perp)^\circ=N_\xi^\perp$, since as we saw $I_\xi={N_\xi}^\circ$. Hence our claim. 
\end{proof}

The fact that the $\b(\xi)$'s are not injective here makes the situation a little more complicated
than in \textsection \ref{sec:concl},
and for example Lemma~\ref{lem:intersect} does not hold. 
It is however
possible to exploit the fact that the $\b(\xi)$'s all have the same rank.

Since $n\geq 2$ we may choose two linearly independent tangent vectors $\xi$ and $\eta$. The letter
$\zeta$ will denote a 
tangent vector in $\la \xi,\eta\ra$.  (In this proof, whenever $(v_i)_{i\in I}$ is a family of
vectors or subspaces in a vector space, $\la v_i,\,i\in I\ra$ denotes the subspace generated by the
$v_i$'s.) 

\begin{lemma}\label{lem:decomp}
There exist decompositions $W_x=W_1\oplus W_2$ and $V_x=V_1\oplus V_2$ such that    
\begin{itemize}
\item $\beta(\zeta)(W_i)\subset V_i$, for all $\zeta$ and all $1\leq i\leq 2$;
\item $\dim V_1=\dim W_1+1$ and $\dim V_2=\dim W_2-1$;
\item $\b(\zeta)_{|W_1}:W_1\pfd V_1$ is one-to-one for all $\zeta\neq 0$; 
\item $\b(\zeta)_{|W_2}:W_2\pfd V_2$ is onto for all $\zeta\neq 0$. 
\end{itemize}
\end{lemma}

\begin{proof}
  The set $\{\b(\zeta),\,\zeta\in \la \xi,\eta\ra\}$ 
is a 2-dimensional linear subspace of $\Hom(W_x,V_x)$, whose non zero elements are all of
rank $(m-1)$ (it is 2-dimensional because $\zeta\mapsto\b(\zeta)$ is linear). Therefore by~\cite{westwick}*{Theorem~3.1}, there exist $r\geq 1$, and decompositions
$W_x=W_0\oplus W_1\oplus\cdots\oplus W_r$ and $V_x=V_0\oplus V_1\oplus\cdots\oplus V_r$ such that 
\begin{itemize}
\item $\beta(\zeta)(W_0)=\{0\}$, for all $\zeta$;
\item $\beta(\zeta)(W_i)\subset V_i$, for all $\zeta$ and all $1\leq i\leq r$;
\item $\dim V_i=\dim W_i\pm 1$ for all $1\leq i\leq r$;
\item $\rk\,\b(\zeta)_{|W_i}=\min\{\dim W_i,\dim V_i\}$, for all $\zeta\neq 0$ and all
  $1\leq i\leq r$. 
\end{itemize}

We remark that since $\xi$ and $\eta$ are linearly independent, the kernels $N_\xi$ and
$N_\eta$ of $\b(\xi)$ and 
$\b(\eta)$ are distinct, and so are their images $I_\xi$ and $I_\eta$. Indeed, the equality of the
images is equivalent 
to the equality of the kernels by skew-symmetry. Therefore if they were equal, we would get that
$\b(\zeta)$ defines an isomorphism $N_\xi^\perp\pfd I_\xi$ for all $\zeta\neq 0$.
This is impossible for example because since $\la\xi,\eta\ra$ is 2-dimensional,
$\zeta\mapsto\det\b(\zeta)$ cannot vanish only for $\zeta=0$.    
Therefore $W_0=\{0\}$. Also, $V_0=\{0\}$ because if not then necessarily $\dim V_0=1$ and
$I_\zeta=V_1\oplus\cdots\oplus V_r$ for all $\zeta\neq 0$. 

Moreover, since $\dim W_x=\dim V_x$,
there must be at least one $i$ such that $\dim V_i=\dim W_i+1$. Say that $\dim V_i=\dim W_i+1$ for
$1\leq i\leq s$ and   $\dim V_i=\dim W_i-1$ for
$s+1\leq i\leq r$. Then for $\zeta\neq 0$, $\rk\b(\zeta)=m-(r-s)=m-s$, so that $s=1$ and
$r=2$. 
\end{proof}

This allows to give the analog of Lemma~\ref{lem:intersect} in the present situation: 

\begin{lemma}\label{lem:intersect2}
We have that $\cap_{\zeta\neq 0}\b(\zeta)W_1=\{0\}$ and $W_2=\la N_\zeta,\,\zeta\neq 0\ra$.
\end{lemma}

\begin{proof}
  For all $\zeta\neq 0$, $\b(\zeta)_{|W_1}$ is injective. Hence the proof of the first statement is
  the same as that of Lemma~\ref{lem:intersect}. By duality, the same reasoning implies the second
  statement.    
\end{proof}

The first point of Lemma~\ref{lem:intersect2} implies that $\g(\eta)$ vanishes on $V_2^\perp$, since
it vanishes on each $I_\zeta^\perp$ by Lemma~\ref{lem:vanish2} and 
$$\la I_\zeta^\perp,\,\zeta\neq 0\ra=\left({\bigcap_{\zeta\neq
      0}I_\zeta}\right)^\perp=\left({\bigcap_{\zeta\neq 
    0}\left({V_2\oplus\b(\zeta)W_1}\right)}\right)^\perp=\left({V_2\oplus\left(\bigcap_{\zeta\neq
    0}\b(\zeta)W_1\right)}\right)^\perp=   V_2^\perp.
$$  
The second point implies that $\g(\eta)(V_2)\subset W_2$, and hence that $\g(\eta)$ vanishes on
$V_2$. Indeed, 
since $\b(\xi)_{|W_2}:W_2\pfd V_2$ is surjective, $V_2$ is generated by vectors of the form
$\b(\xi)w$ with $w\in W_2$ such that $\b(\zeta)w=0$ for some $\zeta\neq 0$. Hence, using the
integrability condition $[\t,\t]=0$ of the Higgs field, we get 
$\b(\zeta)\circ\g(\eta)\circ\b(\xi)w=\b(\eta)\circ\g(\xi)\circ\b(\zeta)w=0$, so that
$\g(\eta)\circ\b(\xi)w\in N_\zeta\subset W_2$. 
Now we saw in Lemma~\ref{lem:vanish2} that $\g(\eta)(V_x)\subset\cap_{\zeta\neq
  0}N_\zeta^\perp=W_2^\perp$. Hence $\g(\eta)(V_2)=\{0\}$.   

We conclude that $\g=0$ outside of $\S_\b$, hence everywhere, so that the $\rho$-equivariant harmonic map $f$ is holomorphic.

\smallskip 

In the same way, if the representation is maximal and $\dego\tilde W<0$, we get that $\b=0$ so that the
$\rho$-equivariant harmonic map 
$f$ is antiholomorphic. 
\end{proof}

\medskip

\begin{proof}[Proof of Proposition~\ref{prop:maxsostar}]

Assume that there exists a holomorphic map $f:\H\pfd\Y$ such that
  $f^\star\o_\Y=\lfloor\frac m2\rfloor\,\o$. 
  By the equality case of the Ahlfors-Schwarz-Pick lemma, see \cite{Ro80}, for all
$\xi\neq 0$ in the  image of $df$, 
the holomorphic sectional curvature of $\o_\Y$ in the direction of $\xi$  is maximal, i.e. equal to
$-\frac 1{\lfloor m/2\rfloor}$. Moreover, the map $f$ is an immersion, so that
  the image of $df$ in $T_\Y$ has dimension $n$ at each point.

The lemma will follow if we prove that for $m\geq 4$, and for $o$ a point in $\Y$,
the maximal dimension of a subspace of 
$T_{\Y,o}$ on which the holomorphic sectional curvatures of $\o_\Y$ equal $-\frac 1{\lfloor m/2\rfloor}$
is $1$. Lemma~\ref{curv:SOstar} and its proof show that this is
equivalent to proving that for $m\geq 4$, the dimension of a maximal linear subspace of
skew-symmetric matrices in $M_m(\C)$ such that $\frac{\tr((A^\star A)^2)}{(\tr(A^\star A))^2}= \frac
1{2\,\lfloor m/2\rfloor}$ 
is $1$. 

The Youla decomposition of a skew-symmetric matrix $A$ (see the proof of
Lemma~\ref{curv:SOstar}) shows that  $\frac{\tr((A^\star A)^2)}{(\tr(A^\star A))^2}= \frac
1{2\,\lfloor m/2\rfloor}$
if and only if 
$A^\star A$ is unitary conjugate to $\alpha^2 1_m$ if $m$ is even (and hence equal to $\alpha^2
1_m$), or to $\alpha^2{\rm diag}(1,\dots,1,0)$ if $m$ is odd. 

\smallskip

We will prove our claim by contradiction. So let now $A$ and $B$ be two linearly independent
skew-symmetric matrix in $M_m(\C)$, such that each non zero
matrix $C$ in the two-dimensional vector space they generate satisfies $\frac{\tr((C^\star
  C)^2)}{(\tr(C^\star C))^2}= \frac 1{2\,\lfloor m/2\rfloor}$. We normalize $A$
and $B$ such that $\tr\bigl(A^\star A\bigr)=\tr\bigl(B^\star B\bigr)=2\lfloor m/2\rfloor$. We can also suppose
that $A$ and $B$ are orthogonal (i.e. $\tr\bigl(A^\star B\bigr)=0$).

If $m$ is even then clearly we have a contradiction. Indeed, for all $\lambda,\mu\in\C$, $(\lambda
A+\mu B)^\star (\lambda A+\mu B)$ is a multiple of $1_m$, hence $\bar\lambda\mu\, A^\star B+\lambda \bar\mu\,
B^\star A$ is also a multiple of $1_m$, but it is trace free so $A^\star B$ must be equal to zero, which is
not possible as the column vectors of $A$ (and $B$) generate $\C^m$. 

From now on, we assume that $m$ is odd. Then for any
$\lambda,\mu\in\C$, $(\lambda,\mu)\neq (0,0)$, the matrix
$$(\lambda A+\mu B)^\star (\lambda A+\mu B)-\frac{1}{m-1}\tr\bigl[(\lambda A+\mu B)^\star (\lambda A+\mu
B)\bigr]1_m$$ 
has rank 1 and since $\tr\bigl(A^\star B\bigr)=0$, $\tr\bigl[(\lambda A+\mu B)^\star (\lambda A+\mu
B)\bigr]=(m-1)(|\lambda|^2+|\mu|^2)$. In other words
$N_{\lambda,\mu}:=M_{\lambda,\mu}-(|\lambda|^2+|\mu|^2)1_m$ has rank $1$, where we denoted $(\lambda
A+\mu B)^\star (\lambda A+\mu B)$ by $M_{\lambda,\mu}$. 

There exists a hyperplane $E$ (resp. $F$) in $\C^m$ such that the endomorphism of $\C^m$ whose matrix is
$A^\star A$ (resp. $B^\star B$) is the identity when restricted to $E$ (resp. $F$). 

If $E=F$ then upon replacing $A$ and $B$ by $U^\top AU$ and $U^\top BU$ for some well chosen $U\in{\rm
  U}(m)$, we may assume that
the $m$-th column vectors of $A$ and $B$ are trivial so that the $m$-th column and the
$m$-th line of $M_{\lambda,\mu}$ both are trivial. As $N_{\lambda,\mu}$ has rank $1$, this implies
that for any $\lambda,\mu$ in $\C$, the upper left $(m-1)$-by-$(m-1)$ block of $M_{\lambda,\mu}$ is equal
to $(|\lambda|^2+|\mu|^2)1_{m-1}$. As in the case when $m$ is even, the upper left  $(m-1)$-by-$(m-1)$ block
of $A^\star B$ should be equal to zero and this is impossible because the column vectors of $A$ and $B$
generate hyperplanes which must intersect non trivially. 

Assume now that $E\cap F$ has codimension $2$. We will use the notation $\la x,y\ra=x^\star y$ for the
standard Hermitian product on $\C^m$ and write $|x|^2=x^\star x$. Let us
denote by $v_1,\dots,v_m$, resp. $w_1,\dots,w_m$, the column vectors of $A$, resp. $B$. Again upon
replacing $A$ and $B$ by $U^\top AU$ and  $U^\top BU$ for some well chosen $U\in{\rm U}(m)$,   
one can assume that $v_m=0$ and that $(v_1,\dots,v_{m-1})$ and $(w_1,\dots,w_{m-2})$ are orthonormal
families.  Moreover,
$w_{m-1}$ and $w_{m}$ are linearly dependent because the bottom right $2$-by-$2$ block of $B^\star B$ must
have determinant $0$. Finally we also have $|w_{m-1}|^2+|w_{m}|^2=1$. 

The bottom right $2$-by-$2$ block of $N_{\lambda,\mu}$ is  
$$
\left(
\begin{matrix} |\mu|^2(|w_{m-1}|^2-1)+\lambda\bar\mu\la w_{m-1},v_{m-1}\ra+\bar\lambda\mu\la
  v_{m-1},w_{m-1}\ra & |\mu|^2\la w_{m-1},w_{m}\ra +\bar \lambda\mu\la v_{m-1},w_{m}\ra\\
|\mu|^2\la w_{m},w_{m-1}\ra +\lambda\bar\mu\la w_{m},v_{m-1}\ra & -|\lambda|^2+|\mu|^2(|w_{m}|^2-1)
\end{matrix} 
\right)
$$
In the determinant of this block, the coefficient of $|\mu|^4$ is equal to $0$ since 
$$(1-|w_{m-1}|^2)(1-|w_{m}|^2)-|\la w_{m},w_{m-1}\ra|^2=|w_{m-1}|^2 |w_{m}|^2-|\la w_{m},w_{m-1}\ra|^2
$$
and $w_{m}$ and $w_{m-1}$ are linearly dependent.

The coefficient of $|\lambda|^2|\mu|^2$ is
$$1-|w_{m-1}|^2-|\la w_{m},v_{m-1}\ra|^2
$$
and by Schwarz inequality, it vanishes if and only if $w_{m}$ and $v_{m-1}$ are proportional since
$|w_{m-1}|^2+|\la w_{m},v_{m-1}\ra|^2\leq |w_{m-1}|^2+|w_{m}|^2=1$. In this case, there exist
complex numbers $a$ and $b$ such that $w_{m-1}=a\,v_{m-1}$, $w_{m}=b\,v_{m-1}$ and $|a|^2+|b|^2=1$. The
above determinant is then equal to $-(|\lambda|^2+|\mu|^2)(\lambda\bar\mu \bar a+\bar\lambda\mu a)$
and vanishes identically if and only if $a=0$. 
 
So $a=0$ and $|b|=1$. This immediately implies that 
$N_{\lambda,\mu}$ is block diagonal with a $(m-2)$-by-$(m-2)$ upper left block and a $2$-by-$2$
bottom right block because $A^\star A$ et $B^\star B$ have the same block decomposition, hence $A^\star B$ and
$B^\star A$ too, since $w_{m-1}=v_{m}=0$ and $w_{m}=b\,v_{m-1}$. The $2$-by-$2$ bottom right block of
$N_{\lambda,\mu}$ is equal to 
$$\left(
\begin{matrix}
-|\mu|^2 & b\bar\lambda\mu\\
\bar b\lambda\bar\mu & -|\lambda|^2
\end{matrix}
\right)
$$
hence has rank $1$ for each $(\lambda,\mu)\not=(0,0)$. As the matrix $N_{\lambda,\mu}$ has rank 1
for all $(\lambda,\mu)\not=(0,0)$, we must have $\la v_i,w_j\ra=0$ for any $1\leq i,j\leq m-2$. If
$m\geq 5$, this is impossible since the two subspaces of codimension 2 generated
respectively by the families $\{v_i\}_{1\leq i\leq m-2}$ and $\{w_i\}_{1\leq i\leq m-2}$ have a
non trivial intersection. 
\end{proof}

\subsection{Representations in Hermitian groups without exceptional factors}\hfill

\smallskip

It is easy to generalize the statement of Theorem~\ref{main} to the case where the lattice $\G$ is
assumed uniform but not torsion free, and the target Lie group $G$ is assumed to be a semisimple Lie
group of Hermitian type without compact or exceptional factors. By this we mean that $G$ is an
almost-direct product of simple noncompact Lie groups of Hermitian type which
are each isogenous to one of the classical groups we have been considering. 

In this case, by Selberg's lemma, there is a normal subgroup $\G'$ of finite index $d$ in $\G$ such
that $\G'$ is torsion free and the
representation $\r'=\r_{|\G'}$ is a product of $k$ representations $\r'_i:\G'\pfd G_i$, where the $G_i$'s are
classical Hermitian noncompact Lie groups. One defines the Toledo invariant of $\rho$ to be $\frac
1d\,\tau(\r')$. Since $\vol(\G\backslash\H)=\frac 1d\,\vol(\G'\backslash\H)$, the representation $\r$
is maximal if and only if the representation $\r'$ is.     
Since $\tau(\rho')=\sum_{i=1}^k\tau(\rho_i')$ and $\rk_\R\, G=\sum_{i=1}^k\rk_\R\, G_i$, $\r'$ is
maximal if and only if each $\r'_i$ is. Therefore in this case $G_i=\SU(p_i,q_i)$ with $p_i\geq
q_in$ for all $i$ and there is a $\r$-equivariant holomorphic or antiholomorphic
map from $\H$ to the symmetric space $\Y=\Pi_{i=1}^k\Y_{p_i,q_i}$ associated to $G$.


\begin{bibdiv}
\begin{biblist}

\bib{BMQ}{article}{
author={Bogomolov, F. A.}, author={McQuillan, M. L.},
title={Rational curves on foliated varieties},
journal={IHES preprint},
date={2001},
}

\bib{BGPG1}{article}{
author={Bradlow, S.~B.}, author={Garcia-Prada, O.}, author={Gothen, P.~B.},
title={Surface group representations and ${\rm U}(p,q)$-Higgs bundles},
journal={J. Diff. Geom.},
volume={64},
date={2003},
pages={111--170},
}

\bib{BGPG2}{article}{
author={Bradlow, S.~B.}, author={Garcia-Prada, O.}, author={Gothen, P.~B.},
title={Maximal surface group representations in isometry groups of classical Hermitian
  symmetric spaces},
journal={Geom. Dedicata},
volume={122},
date={2006},
pages={185--213},
}

\bib{BH}{book}{ 
author={Bridson, M.~R.}, 
author={Haefliger, A.}, 
title={Metric spaces of non-positive curvature},
series={Grundlehren der Mathematischen Wissenschaften},
volume={319},
publisher={Springer-Verlag},
address={Berlin},
date={1999},
}

\bib{BI07}{article}{
author={Burger, M.}, 
author={Iozzi, A.},
title={Bounded differential forms, generalized Milnor-Wood inequality and an application to
  deformation rigidity},
journal={Geom. Dedicata},
volume={125},
date={2007},
pages={1--23},
}

\bib{BI08}{article}{
author={Burger, M.}, 
author={Iozzi, A.}, 
title={A measurable Cartan theorem and applications to deformation rigidity in complex
  hyperbolic geometry},
journal={Pure Appl. Math. Q.},
volume={4},
date={2008},
pages={181--202},
}

\bib{BIW09}{article}{ 
author={Burger, M.}, 
author={Iozzi, A.}, 
author={Wienhard, A.},
title={Tight homomorphisms and Hermitian symmetric spaces},
journal={Geom. Funct. Anal.},
volume={19},
date={2009},
pages={678--721},
}


\bib{BIW}{article}{ 
author={Burger, M.}, 
author={Iozzi, A.}, 
author={Wienhard, A.},
title={Surface group representations with maximal Toledo invariant},
journal={Ann. of Math.},
volume={172},
date={2010},
pages={517--566},
}

\bib{CW}{article}{
author={Cao, J.-G.}, 
author={Wong, P.-M.}, 
title={Finsler geometry of projectivized vector bundles},
journal={J. Math. Kyoto Univ. },
volume={43},
date={2003},
pages={369--410},
}

\bib{CorletteFlatGBundles}{article}{
author={Corlette, K.}, 
title={Flat $G$-bundles with canonical metrics},
journal={J. Diff.  Geom.},
volume={28},
date={1988},
pages={361--382},
} 


\bib{CorletteArchimedean}{article}{
author={Corlette, K.}, 
title={Archimedean superrigidity and hyperbolic geometry},
journal={Ann. of Math.},
volume={135},
date={1992},
pages={165--182},
}

\bib{Dani}{article}{
author={Dani, S.}, 
title={A simple proof of Borel's density theorem},
journal={Math. Z.},
volume={174},
date={1980},
pages={81--94},
}

\bib{Eberlein}{book}{
author={Eberlein, P.}, 
title={Geometry of nonpositively curved manifolds},
series={Chicago Lectures in Mathematics},
publisher={University of Chicago Press},
address={Chicago, IL},
date={1996},
}

\bib{Fischer}{book}{
author={Fischer, G.}, 
title={Complex analytic geometry},
series={Lecture Notes in Mathematics},
publisher={Springer-Verlag},
address={Berlin-New York},
volume={538},
date={1976},
}

\bib{God}{book}{
author={Godbillon, C.}, 
title={Feuilletages. \'Etudes g\'eom\'etriques},
series={Progress in Mathematics},
publisher={Birkh\"auser Verlag},
address={Basel},
volume={98},
date={1991},
}

\bib{GoldmanThesis}{thesis}{
author={Goldman, W.~M.}, 
title={Discontinuous groups and the Euler class},
address={University of California at Berkeley},
type={Ph.D. Thesis},
date={1980},
}

\bib{GoldmanComponents}{article}{
author={Goldman, W.~M.}, 
title={Topological components of spaces of representations},
journal={Invent. Math.},
volume={93},
date={1988},
pages={557--607},
}

\bib{GoldmanMillsonLocalRigidity}{article}{
author={Goldman, W.~M.}, 
author={Millson, J.~J.}, 
title={Local rigidity of discrete groups acting on complex hyperbolic space},
journal={Invent. Math.},
volume={88},
date={1987},
pages={495--520},
}

\bib{GriffithsSchmid}{article}{
author={Griffiths, P.},
author={Schmid, W.},
title={Locally homogeneous complex manifolds},
journal={Acta Mathematica},
date={1969},
volume={123},
pages={253--302},
}



\bib{Gromov1}{article}{
author={Gromov, M.},
title={Foliated Plateau problem, Part I: Minimal varieties},
journal={Geom. Funct. Anal.},
volume={1},
date={1991},
pages={14--79},
}

\bib{Gromov2}{article}{
author={Gromov, M.},
title={Foliated Plateau problem, Part II: Harmonic maps of foliations},
journal={Geom. Funct. Anal.},
volume={1},
date={1991},
pages={253--320},
}


\bib{GPS}{article}{
author={Gromov, M.}, 
author={Piatetski-Shapiro, I.}, 
title={Nonarithmetic groups in Lobachevsky spaces},
journal={Inst. Hautes Études Sci. Publ. Math.},
volume={66},
date={1988},
pages={93--103},
}

\bib{GromovSchoen}{article}{
author={Gromov, M.}, 
author={Schoen, R.}, 
title={Harmonic maps into singular spaces and {$p$}-adic superrigidity for lattices in groups of rank one},
journal={Inst. Hautes Études Sci. Publ. Math.},
volume={76},
date={1992},
pages={165--246},
}

\bib{GW}{article}{
author={Guichard, O.}, 
author={Wienhard, A.}, 
title={Anosov representations: domains of discontinuity and applications},
journal={ Invent. Math.},
volume={190},
date={2012},
pages={357--438},
}

\bib{Hernandez}{article}{
author={Hernandez, L.}, 
title={Maximal representations of surface groups in bounded symmetric domains},
journal={Trans. Amer. Math. Soc.},
volume={324},
date={1991},
pages={405--420},
}

\bib{JohnsonMillson}{article}{
  author={Johnson, D.}, 
  author={Millson, J.~J.}, 
  title={Deformation spaces associated to compact hyperbolic manifolds},
  book={
    title={Discrete Groups in Geometry and Analysis ({N}ew {H}aven, {C}onn., 1984)},
    series={Progress in  Mathematics},
    publisher={Birkh\"auser},
    address={Boston, MA},
    volume={67},
  },
  date={1987},
  pages={48--106},  
}

\bib{Klingler}{article}{
author={Klingler, B.}, 
title={Local rigidity for complex hyperbolic lattices and Hodge theory},
journal={Invent. Math.},
volume={184},
date={2011},
pages={445--498},
}

\bib{K}{book}{
author={Kobayashi, S.}, 
title={Differential geometry of complex vector bundles},
publisher={Princeton University Press},
date={1987},
}

\bib{KozMobRank1}{article}{
author={Koziarz, V.}, 
author={Maubon, J.}, 
title={Harmonic maps and representations of non-uniform lattices of {${\rm PU}(m,1)$}},
journal={Annales de l'Institut Fourier (Grenoble)},
volume={58},
date={2008},
pages={507--558},
}

\bib{KozMobRank2}{article}{
author={Koziarz, V.}, 
author={Maubon, J.}, 
title={Representations of complex hyperbolic lattices into rank 2 classical Lie Groups of
  Hermitian type},
journal={Geom. Dedicata},
volume={137},
date={2008},
pages={85--111 },
}

\bib{KozMobequidistrib}{article}{
author={Koziarz, V.}, 
author={Maubon, J.}, 
title={On the equidistribution of totally geodesic submanifolds in locally symmetric spaces and application to boundedness results for negative curves and exceptional divisors},
journal={ arXiv:1407.6561},
date={2014},
}


\bib{Margulis}{book}{
author={Margulis, G.~A.}, 
title={Discrete subgroups of semisimple Lie groups},
series={Ergebnisse der Mathematik und ihrer Grenzgebiete},
volume={17},
date={1991},
publisher={Springer-Verlag},
address={Berlin},
}

\bib{MarkmanXia}{article}{
author={Markman, E.}, 
author={Xia, E.~Z.}, 
title={The moduli of flat {${\rm PU}(p,p)$}-structures with large {T}oledo invariants},
journal={Math. Z.},
volume={240},
date={2002},
pages={95--109},
}

\bib{Maubon}{article}{
author={Maubon, J.}, 
title={Higgs bundles and representations of complex hyperbolic lattices},
book={
title={Handbook of group actions, vol. II},
series={Advanced Lectures in Mathematics},
publisher={International Press and Higher Education Press},
volume={32},
},
date={2015},
pages={201--244},
}

\bib{Mok}{article}{
author={Mok, N.}, 
title={On holomorphic immersions into K\"ahler manifolds of constant holomorphic sectional curvature},
journal={Sci. China Ser. A},
volume={48},
date={2005},
pages={123--145 },
}
 
\bib{P}{article}{
author={Payne, T.}, 
title={Closures of totally geodesic immersions into locally symmetric spaces of noncompact
  type},
journal={Proc. Amer. Math. Soc.},
volume={127},
date={1999},
pages={829--833},
}

\bib{Per}{article}{
author={Pereira, J. V.},
title={Global stability for holomorphic foliations on K\"ahler manifolds},
journal={Qual. Theory Dyn. Syst.},
volume={2},
date={2001},
pages={381--384}
}

\bib{Pop}{article}{
author={Popovici, D.},
title={A simple proof of a theorem by Uhlenbeck and Yau},
journal={Math. Z.},
volume={250},
date={2005},
pages={855--872},
}

\bib{Pozzetti}{article}{
author={Pozzetti, M.~B.}, 
title={Maximal representations of complex hyperbolic lattices in {${\rm SU}(m,n)$}},
journal={Geom. Funct. Anal.},
volume={25},
date={2015},
pages={1290--1332},
}

 

\bib{Raghunathan}{book}{
author={Raghunathan, M.~S.}, 
title={Discrete subgroups of Lie groups},
series={Ergebnisse der Mathematik und ihrer Grenzgebiete},
volume={68},
publisher={Springer-Verlag},
address={New York-Heidelberg},
date={1972},
}

\bib{RatnerTopo}{article}{
author={Ratner, M.}, 
title={Raghunathan’s topological conjecture and distributions of unipotent flows},
journal={Duke Math. J.},
volume={63},
date={1991},
pages={235--280},
}

\bib{RatnerMeasure}{article}{
author={Ratner, M.}, 
title={On Raghunathan's measure conjecture},
journal={Ann. of Math.},
volume={134},
date={1991},
pages={545--607},
}

\bib{Richardson}{article}{
author={Richardson, R. W.}, 
title={Conjugacy classes of $n$-tuples in Lie algebras and algebraic groups},
journal={Duke Math. J.},
volume={57},
date={1988},
pages={1--35},
}


\bib{Ro80}{article}{
author={Royden, H.~L.}, 
title={The Ahlfors-Schwarz lemma in several complex variables},
journal={Comment. Math. Helvetici},
volume={55},
date={1980},
pages={547--558},
}


\bib{Sampson}{article}{
   author={Sampson, J.~H.},
   title={Applications of harmonic maps to K\"ahler geometry},
   conference={
      title={Complex differential geometry and nonlinear differential
      equations },
      address={Brunswick, Maine},
      date={1984},
   },
   book={
      series={Contemp. Math.},
      volume={49},
      publisher={Amer. Math. Soc., Providence, RI},
   },
   date={1986},
   pages={125--134},
}

\bib{Shah}{article}{
author={Shah, N.}, 
title={Uniformly distributed orbits of certain flows on homogeneous spaces}, 
journal={Math. Ann.},
volume={289},
date={1991},
pages={315--334},
}

\bib{Sibley}{article}{
author={Sibley, B.}, 
title={Asymptotics of the Yang–Mills flow for holomorphic vector bundles over K\"ahler manifolds: The
  canonical structure of the limit},
journal={J. reine angew. Math., Ahead of Print, DOI 10.1515/ crelle-2013-0063},
date={2013},
pages={},
}

\bib{S1}{article}{
author={Simpson, C.}, 
title={Constructing variations of Hodge structure using Yang-Mills theory and applications to
uniformization},
journal={J. Amer. Math. Soc.},
volume={1},
date={1988},
pages={867--918},
}

\bib{S2}{article}{
author={Simpson, C.}, 
title={Higgs bundles and local systems},
journal={Inst. Hautes \'Etudes Sci. Publ. Math.},
volume={75},
date={1992},
pages={5--95},
}

\bib{Siu}{article}{
author={Siu, Y.-T.}, 
title={The complex-analyticity of harmonic maps and the strong rigidity of compact K{\"a}hler
  manifolds},
journal={Ann. of Math.},
volume={112},
date={1980},
pages={73--111},
}

\bib{Sullivan}{article}{
author={Sullivan, D.},
title={Cycles for the dynamical study of foliated manifolds and complex manifolds},
journal={Invent. Math.},
volume={36},
date={1976},
pages={225--255},
}



\bib{ToledoHarmonic}{article}{
author={Toledo, D.}, 
title={Harmonic maps from surfaces to certain {K}aehler manifolds},
journal={Mathematica Scandinavica},
volume={45},
date={1979},
pages={13--26},
}

\bib{ToledoRepresentations}{article}{
author={Toledo, D.}, 
title={Representations of surface groups in complex hyperbolic space},
journal={J. Diff. Geom.},
volume={29},
date={1989},
pages={125--133},
}

\bib{UhlenbeckYau}{article}{
   author={Uhlenbeck, K.},
   author={Yau, S.-T.},
   title={On the existence of Hermitian-Yang-Mills connections in stable
   vector bundles},
   journal={Comm. Pure Appl. Math.},
   volume={39},
   date={1986},
   pages={S257--S293},
   }
	


\bib{westwick}{article}{
author={Westwick, R.}, 
title={Spaces of linear transformations of equal rank},
journal={Linear Algebra and Appl.},
volume={5},
date={1972},
pages={49--64},
}

\bib{Wolf}{article}{
author ={Wolf, J. A.},
journal={Bull. Amer. Math. Soc.},
pages={1121--1237},
title={The action of a real semisimple group on a complex flag manifold. I: Orbit structure and holomorphic arc components},
volume={75},
year={1969}
}


\bib{Xia}{article}{
author={Xia, E.~Z.}, 
title={The moduli of flat {${\rm PU}(2,1)$} structures on {R}iemann surfaces},
journal={Pacific J. Math.},
volume={195},
date={2000},
pages={231--256},
}

\bib{You}{article}{
author={Youla, D.~C.}, 
title={A normal form for a matrix under the unitary congruence group},
journal={Canad. J. Math.},
volume={13},
date={1961},
pages={694--704}, 
}



\end{biblist}
\end{bibdiv}
\end{document}